\documentclass[11pt,reqno]{amsart}
\usepackage{amsmath,amssymb,mathrsfs}
\usepackage{graphicx,cite}
\usepackage{subfigure}
\usepackage{epstopdf}
\usepackage{graphics} %% add this and next lines if pictures should be in esp format
\usepackage{epsfig} %For pictures: screened artwork should be set up with an 85 or 100 line screen

\newcommand{\R}{{\mathbb R}}
\newcommand{\C}{{\mathbb C}}
\newcommand{\be}{\begin{eqnarray}}
\newcommand{\ben}{\begin{eqnarray*}}
\newcommand{\en}{\end{eqnarray}}
\newcommand{\enn}{\end{eqnarray*}}

\newtheorem{theorem}{Theorem}[section]
\newtheorem{theo}[theorem]{Theorem}
\newtheorem{lemm}[theorem]{Lemma}
\newtheorem{corollary}[theorem]{Corollary}
\newtheorem{defi}[theorem]{Definition}
\newtheorem{rema}[theorem]{Remark}

\numberwithin{equation}{section}

\begin{document}

\title[elastic scattering from rough surfaces]{Elastic scattering from rough surfaces in three dimensions}

\author{Guanghui Hu}
\address{Beijing Computational Science Research Center, Beijing 100193, China}
\email{hu@csrc.ac.cn}

\author{Peijun Li}
\address{Department of Mathematics, Purdue University, West Lafayette, Indiana 47907, USA}
\email{lipeijun@math.purdue.edu}

\author{Yue Zhao}
\address{School of Mathematics and Statistics, Central China Normal University, Wuhan 430079, China}
\email{zhaoyueccnu@163.com}

%\thanks{}



\begin{abstract}
Consider the elastic scattering of a plane or point incident wave by an unbounded and rigid rough surface. The angular spectrum representation (ASR) for the time-harmonic Navier equation is derived in three dimensions. The ASR is utilized as a radiation condition to the elastic rough surface scattering problem. The uniqueness is proved through a Rellich-type identity for surfaces given by uniformly Lipschitz functions. In the case of flat surfaces with a local perturbation, we deduce an equivalent variational formulation in a truncated bounded domain and show the existence results for general incoming waves. The main ingredient of the proof is the radiating behavior of the Green tensor to the first boundary value problem of the Navier equation in a half space.
\end{abstract}

\maketitle

\section{Introduction}

Rough surface scattering problems have important applications in diverse scientific areas such as remote sensing, geophysics, outdoor sound propagation, radar techniques. Significant progress has been made by Chandler-Wilde and his co-authors concerning the mathematical analysis and the numerical approximation of the acoustic scattering problems modeled by the Helmholtz equation. We refer to \cite{CR, CW-Zhang-98b, CHP2006, S.B.1999,CW-Zhang-98a} for the integral equation method and to \cite{CWEL,cm-siaa05} the variational approach in both two and three dimensional settings. In the work of Duran, Muga and Nedelec \cite{Nedelec09}, the radiation condition and well-posedness in the absence of acoustic surfaces waves were discussed under the non-absorbing boundary condition in a locally perturbed half plane. The electromagnetic scattering problems were studied in \cite{lwz12} when the medium is lossy and also in \cite{Haddar, lzz-m2as17} in the more challenging case of a penetrable dielectric layer.

This paper concerns the mathematical analysis of the time-harmonic elastic scattering from unbounded rigid surfaces
in three dimensions. The relevant phenomena for the elastic wave propagation can be found in geophysics and seismology (see e.g.,\cite{Abubakar,A73} and the references cited therein). In linear elasticity, the  existence and uniqueness of
solutions were firstly given by Arens in \cite{Aren99, a-siaa01,Aren02} for
$C^{1,\alpha}$-smooth rough surfaces via the boundary integral
equation method in two dimensions, which generalize the solvability results of
\cite{CW-Zhang-98b, S.B.1999,CW-Zhang-98a}
for acoustic waves to elastic waves. Moreover, an upward propagating radiation condition (UPRC) was proposed
in \cite{a-siaa01} based on the elastic Green's tensor of the Dirichlet
boundary value problem for the Navier equation in a
half-space. It is known that the classical Kupradze radiation condition
(e.g. \cite{Ku}) is not appropriate in the case of unbounded rough
surfaces. The variational approach was firstly proposed in \cite{eh-mmmas12, eh-2010} for treating the well-posedenss in periodic structures with the Rayleigh expansion condition (REC) and in \cite{eh-siaa12,eh-2015} for general rigid rough surfaces using the angular spectrum representation (ASR) (see also \cite{radiation} for early discussions with less rigourous arguments).
However, most of these works are devoted to two-dimensional elastic scattering problems and little analysis has been carried out in three dimensions.

The goal of this paper is threefold. First, we present a mathematical setting of the elastic rough surface scattering problems in three dimensions. In particular, we derive the upward angular spectrum representation (UASR) and the Green's tensor to the first boundary value problem of the Navier equation in a half space. To be best of our knowledge, the UASR and the Green's tensor have not been rigorously investigated in the mathematical literature. The UASR for the Navier equation can be regarded as a formal outgoing radiation condition in rough surface scattering problems (see \cite{cm-siaa05} in the acoustic case) and it leads to an equivalent Dirichlet-to-Neumann map (transparent boundary condition) which can be used to truncate the unbounded domain in the vertical direction. Second, we prove the uniqueness of weak solutions if the rigid surface is the graph of a uniformly Lipschitz continuous function. As in the two-dimensional case \cite{eh-siaa12}, our uniqueness proof is essentially based on a Rellich-type identity in an unbounded strip. However, the calculation of some key integral identities (see e.g., \eqref{eq:4}) is much more involved than the two dimensional problem. Third, as an application of the half-space radiation condition and Green's tensor, we show the existence of solutions to locally perturbed scattering problems. Unlike the Helmholtz or Maxwell equations (see e.g., \cite{baohuyin18,Wood2006,lwz12,Li2010, Li2015}, an essential difficulty in elasticity arises from the lack of a series solution of the Navier equation satisfying the Dirichlet boundary condition on the ground plane. We refer to Remarks \ref{Rem:5.4} and \ref{rema:5.7} for a detailed comparision of the well-posedness results presented in this paper and those in acoustic and electromagnetic waves.
The local perturbation argument has significantly simplified the analysis for general rough surfaces, because one can derive an equivalent variational formulation in a bounded domain in which the Fredhlom alternative can be applied. We point out some open questions in this respective in  Section \ref{sec:con}. Our future work will be devoted to well-posedness of general (non-periodic) rough surface scattering problems.

It should be remarked that, elastic surface waves, which exponentially decay in the vertical direction,  fulfill the newly established radiation condition (\ref{uprc}) in a weighted Sobolev space (see e.g., \cite{eh-2015} in 2D) rather than the usual $H^1$-space as considered in this paper. Hence,
our uniqueness result (Theorem \ref{uniqueness}) does not give rise to the absence of surface waves caused by a rigid scattering interface. In fact, the horizontally decaying behavior of solutions in $H^1$ (see Theorem \ref{uniqueness}) excludes elastic surface waves.
A possible future effort is to analyze the absence of elastic surface waves by proving well-posedness in weighted Sobolev spaces, if the rigid rough surface is the graph of a function.
  For flat surfaces with local perturbations, the well-posedness results and the solution form (see Theorems \ref{exist} and \ref{th5.12})  are not valid under the traction-free boundary condition, due to the presence of surface waves in the far-field expansion. We refer to \cite{DMN2011} for the two-dimensional Green's tensor with a free flat boundary and the corresponding well-posedness result in a locally perturbed half-plane. The limiting absorption principle was justified in \cite{DG1988}
for a free boundary in a locally perturbed half space.
Note that our arguments for rigid flat surfaces with local perturbations  depend on the asymptotic behaviour of the half-space Green's tensor which is different from the case of free boundaries; see Theorems \ref{exist} and \ref{th5.12}.

The remaining part of this paper is organized as follows. In Section \ref{sec:2}, we formulate the three-dimensional rough surface problems and introduce the upward and downward angular spectrum representations. The downward and upward Dirichlet-to-Neumann maps will be defined and analyzed in Section \ref{sec:3}. Section \ref{cuniqueness} is devoted to the uniqueness proof for general rough surface scattering problems; while Section \ref{existence} is devoted to the existence for locally perturbed scattering problems. Some concluding remarks and open questions will be presented in the Section \ref{sec:con}.

\section{Problem formulation}\label{sec:2}
\setcounter{equation}{0}

In this section, we present the mathematical formulation of the three-dimensional elastic wave scattering by unbounded rigid rough surfaces. Let $D \subset \mathbb R^3$ be an unbounded connected open set such that, for some constants $f_-<f_+$,
\[
U_{f_+}\subset D \subset U_{f_-}, \quad U_b : = \{x= (x^{\prime},x_3): x_3>b\}, \quad x^{\prime}:=(x_1, x_2).
\]
For $b>f_+$, let $\Gamma_b= \{x\in \mathbb R^3: x_3 = b\}$ and $S_b= D \backslash \overline{U}_b$. We assume that $\Gamma:= \partial D$ is an unbounded rough surface, which is Lipschitz continuous but not necessary the graph of some function. The space $D$ is supposed to be filled with a homogeneous and isotropic elastic medium with unit mass density.

Let ${u}^{\rm in}$ be a time-harmonic elastic wave which is incident on the rough surface from above. Let $\omega>0$ be the angular frequency of the  incident wave. Denote by $\lambda$, $\mu$ the Lam\'{e} constants characterizing the medium above $\Gamma$ and satisfying $\mu>0, \lambda +2\mu/3>0$. The incident wave $u^{\rm in}$ is allowed to be a general elastic plane wave of the following form
 \begin{align}\label{inc}
{u}^{\rm in}(x) = c_{\rm p}{u}^{\rm in}_{\rm p}(x) +
c_{\rm s, 1}{u}^{\rm in}_{{\rm s}, 1}(x)
 +  c_{\rm s, 2}{u}^{\rm in}_{{\rm s}, 2}(x), \quad c_{\rm p}, c_{\rm s, j} \in \mathbb C,\, j=1,2,
\end{align}
where ${u}^{\rm in}_{\rm p}$ is the compressional plane wave
 \begin{align}\label{icw}
 {u}^{\rm in}_{\rm p}(x)={d} e^{{\rm
i}\kappa_{\rm p}{x}\cdot{d}}, \quad   d:=d(\theta, \varphi)=(\sin\theta\cos\varphi, \sin\theta\sin\varphi, -\cos\theta)^\top
\end{align}
and ${u}^{\rm in}_{{\rm s}, j}$ are the shear plane waves
\begin{align}\label{isp}
 {u}^{\rm in}_{{\rm s}, j}(x)=d_j^\bot e^{{\rm
i}\kappa_{\rm s}{x}\cdot{d}},\quad j=1,2.
\end{align}
Here $\theta\in [0, \pi/2), \varphi\in[0, 2\pi)$ are the incident angles, $d_j^\bot$ are unit vectors satisfying $d^\bot_j \cdot d = 0$, and
\begin{equation*}
\kappa_{\rm p}=\omega/\sqrt{\lambda+2\mu}, \quad \kappa_{\rm s}=\omega/\sqrt{\mu}
\end{equation*}
are the compressional and shear wavenumbers, respectively. It is clear to note that ${u}^{\rm in}_{\rm p}$ is a longitudinal wave and ${u}^{\rm in}_{{\rm s}, j}, j=1, 2$ are transversal waves.
 It can be verified that the incident field ${u}^{\rm in}$ satisfies the three-dimensional time-harmonic Navier equation:
\begin{equation}\label{ifne}
\mu\Delta{u}^{\rm
in}+(\lambda+\mu)\nabla \nabla\cdot{u}^{\rm in} +
\omega^2{u}^{\rm in}=0\quad\text{in}~ \R^3.
\end{equation}

In this paper, we assume that the elastic medium beneath the rough surface is impenetrable and rigid. Hence the total field satisfies the homogeneous Dirichlet boundary condition
\begin{equation*}
 {u}=0\quad\text{on}~ \Gamma.
\end{equation*}
Clearly, the displacement of the scattered field $u^{\rm sc}:=u-u^{\rm in}$ satisfies the following boundary value problem
\be\label{tf}
 \mu\Delta{u^{\rm sc}}+(\lambda+\mu)\nabla
\nabla\cdot{u^{\rm sc}} + \omega^2{u^{\rm sc}}=0\quad\text{in}~ D,\qquad
u^{\rm sc}=-u^{\rm in}\quad\mbox{on}\quad\Gamma.
\en

We may also consider a spherical point source incidence  given by the Green tensor of the Navier equation in $\R^3$, i.e.,
\be\label{source-incidence}
u^{\rm in}(x) = \text{G}( x, y),\quad x\in D\backslash\{y\}, \quad y\in D,
\en
where
\begin{equation}\label{gtn}
  \text{G}( x, y)=\frac{1}{\mu}g_{\rm s}( x, y)\text{I}
 +\frac{1}{\omega^2}\nabla_{y}\nabla^\top_{
y}(g_{\rm s}(x, y)-g_{\rm p}( x,
y)).
\end{equation}
Here $\text{I}$ is the identity matrix and
\begin{equation}\label{gn}
g_{\rm p}(x, y)=\frac{1}{4\pi}\frac{e^{{\rm
i}\kappa_{\rm p}|x-y|}}{|x-y|}, \quad g_{\rm s}(x, y)=\frac{1}{4\pi}\frac{e^{{\rm
i}\kappa_{\rm s}|x-y|}}{|x-y|}
\end{equation}
are the fundamental solutions of the three dimensional Helmholtz equations with the compressional and shear wave numbers, respectively. The incident field  (\ref{source-incidence}) satisfies the three dimensional Navier equation:
\begin{align*}
\mu\Delta{u}^{\rm
in}+(\lambda+\mu)\nabla \nabla\cdot{u}^{\rm in} +
\omega^2{u}^{\rm in}=\delta(x-y)\text{I}\quad\text{in}~ \mathbb R^3\backslash\{y\}.
\end{align*}

Since the domain $D$ is unbounded, a radiation condition must be imposed at infinity to ensure the well-posedness of the boundary value problem \eqref{tf}. Following \cite{eh-siaa12}, we propose a radiation condition based on the upward angular spectrum representation (UASR) for solutions of the scalar Helmholtz equation \cite{cm-siaa05}.

We begin with the decomposition of the scattered field into a sum of its compressional and shear parts
\begin{align}\label{hd}
u^{\rm sc}  = \frac{1}{\rm i}(\nabla \varphi + \nabla \times  \psi), \quad \nabla \cdot  \psi = 0,
\end{align}
where the scalar function $\varphi$ and the vector function $\psi$ satisfy the homogeneous Helmholtz equations
\begin{equation*}
 \Delta\varphi+\kappa^2_{\rm p} \varphi=0, \quad \Delta  \psi+\kappa^2_{\rm s}  \psi=0 \quad \text{in}~D.
\end{equation*}

Denote by $\hat{v}$ the Fourier transform of $v$ in $\mathbb R^2$, i.e.,
\[
\hat{v}( \xi) =\mathcal{F}v(\xi):=\frac{1}{2\pi} \int_{\mathbb R^2} v(x^{\prime}) e^{-{\rm i}x^{\prime}\cdot  \xi}{\rm d}x^{\prime},\quad \xi=(\xi_1, \xi_2)\in \R^2.
\]
Taking the Fourier transform of \eqref{hd} and assuming that $\varphi, \psi$ fulfill the UASR for the Helmholtz equations in $U_b$, we obtain
\be\label{re}
\varphi(x^{\prime},x_3)=\frac{1}{2\pi} \int_{\mathbb R^2} \hat{\varphi} ( \xi,b) e^{{\rm i}\beta ( \xi)(x_3-b)}
e^{{\rm i} \xi \cdot x^{\prime}}{\rm d} \xi, \quad
 \psi(x^{\prime},x_3)= \frac{1}{2\pi}\int_{\mathbb R^2} \hat{ \psi} ( \xi,b) e^{{\rm i}\gamma ( \xi)(x_3-b)}
e^{{\rm i} \xi \cdot x^{\prime}}{\rm d} \xi,
\en
where
\begin{align*}
 \beta( \xi):=
 \begin{cases}
   (\kappa^2_{\rm p} - | \xi|^2)^{1/2},\quad
&| \xi|<\kappa_{\rm p},\\
  {\rm i} (| \xi|^2 - \kappa^2_{\rm p})^{1/2},\quad
&| \xi|>\kappa_{\rm p},
 \end{cases}
\end{align*}
and
\begin{align*}
 \gamma( \xi):=
 \begin{cases}
   (\kappa^2_{\rm s} - | \xi|^2)^{1/2},\quad
&| \xi|<\kappa_{\rm s},\\
  {\rm i} (| \xi|^2 - \kappa^2_{\rm s})^{1/2},\quad
&| \xi|>\kappa_{\rm s}.
 \end{cases}
\end{align*}

Denote
\begin{align*}
A_{{\rm p}}(\xi) = \hat{\varphi} ( \xi,b), \quad \tilde{\boldsymbol A_{{\rm s}}}(\xi) = \hat{ \psi} ( \xi,b).
\end{align*}
Substituting \eqref{re} into \eqref{hd}, we obtain
\begin{align}\label{uprc1}
u^{\rm sc}(x) =\frac{1}{2\pi} \int_{\mathbb R^2} \left[ A_{\rm p}(\xi)\,(\xi,\beta)^{\top} e^{{\rm i}\beta (x_3-b)}
 + \boldsymbol A_{{\rm s}}(\xi)\, e^{{\rm i}\gamma (x_3-b)}\right]
e^{{\rm i} \xi \cdot x^{\prime}}{\rm d} \xi,
\end{align}
where  $\boldsymbol A_{{\rm s}} = (A^{(1)}_{\rm s}, A^{(2)}_{\rm s}, A^{(3)}_{\rm s})^{\top}(\xi):=(\xi,\gamma)^{\top}\times \tilde{\boldsymbol A_{{\rm s}}}(\xi)$. It follows from \eqref{uprc1} and the orthogonality $(\xi, \gamma)\cdot \boldsymbol A_{\rm s}^{\top}=0$ that
\begin{align*}
\begin{bmatrix}
\hat{u}^{\rm sc}(\xi, b) \\ 0
\end{bmatrix}
 =
\begin{bmatrix}
\xi_1 & 1 & 0 & 0\\
\xi_2 & 0 & 1 & 0\\
\beta & 0 & 0 & 1 \\
0 & \xi_1 & \xi_2 & \gamma
\end{bmatrix}
\begin{bmatrix}
A_{\rm p}(\xi)\\[5pt]
\boldsymbol A_{\rm s}^{\top}(\xi)
\end{bmatrix}
:=\tilde{\mathbb{D}}(\xi)\,\boldsymbol A(\xi),
\end{align*}
which gives
\be\label{A}
\boldsymbol A(\xi)=\begin{bmatrix}
A_{\rm p}\\[5pt]
\boldsymbol A_{\rm s}^{\top}
\end{bmatrix}(\xi)=\tilde{\mathbb{D}}^{-1}(\xi)\,
\begin{bmatrix}\hat{u}^{sc}(\xi,b) \\ 0 \end{bmatrix}
=\mathbb{D}(\xi)\, \hat{u}^{\rm sc}(\xi,b).
\en
Here $\mathbb{D}$ is a $4\times 3$ matrix given by
\ben
\mathbb{D}(\xi)=\frac{1}{\beta \gamma+|\xi|^2}\begin{bmatrix}
\xi_1 & \xi_2 & \gamma \\
\beta \gamma+\xi_2^2 & -\xi_1\xi_2 & -\xi_1\gamma \\
-\xi_1\xi_2 & \beta\gamma+\xi^2 & -\xi_2 \gamma \\
-\xi_1\beta & -\xi_2 \beta & |\xi|^2
\end{bmatrix}.
\enn

Using \eqref{uprc1}--\eqref{A} yields an expression of $u^{\rm sc}$ in $U_b$ in terms of the Fourier transform of the Dirichlet data $u(x', b)$:
\begin{align}\label{uprc}
 u^{\rm sc}(x)=\frac{1}{2\pi}\int_{\mathbb R^2}\Big\{
\frac{1}{\beta\,\gamma+| \xi|^2}  \Big(M_{\rm p}(\xi) e^{{\rm i} (\xi\cdot x^{\prime}+ \beta (x_3-b))} + M_{\rm s}(\xi) e^{{\rm i} (\xi\cdot x^{\prime}+\gamma (x_3-b))}  \Big)   \hat{ u}^{\rm sc}(\xi, b)
  \Big\} {\rm d} \xi,
\end{align}
where
\be\label{Mp}
M_{\rm p}(\xi)=(\xi_1,\xi_2,\beta)\otimes (\xi_1, \xi_2,\gamma):=\begin{bmatrix}
\xi_1^2        &  \xi_1\xi_2    & \xi_1 \gamma\\
\xi_1\xi_2     & \xi_2^2        & \xi_2 \gamma\\
\xi_1\beta  & \xi_2\beta  & \beta \gamma
\end{bmatrix}
\en
and
\be\label{Ms}
M_{\rm s}(\xi)=
\begin{bmatrix}
\beta\gamma+\xi_2^2 & -\xi_1\xi_2               & -\gamma \xi_1\\
-\xi_1 \xi_2             & \beta\gamma+\xi_1^2  & -\gamma \xi_2\\
-\xi_1\beta           & -\xi_2\beta            & |\xi|^2
\end{bmatrix}=(\beta \gamma+|\xi|^2)\,\textbf{\rm I} -M_{\rm p}(\xi).
\en
Define $M_{\rm p}^+=M_{\rm p}/(\beta\gamma+|\xi|^2)$. We can rewrite \eqref{uprc} into
\begin{align}\label{uprc2}
 u^{\rm sc}(x)=\frac{1}{2\pi}\int_{\mathbb R^2}\Big\{
M^+_{\rm p}(\xi) e^{{\rm i} (\xi\cdot x^{\prime}+ \beta (x_3-b))} + \Big(\textbf{\rm I}-M^+_{\rm p}(\xi)\Big) e^{{\rm i} (\xi\cdot x^{\prime}+\gamma (x_3-b))}  \Big\}\widehat{u}^{\rm sc}(\xi, b)
  {\rm d} \xi.
\end{align}

The representation \eqref{uprc} or \eqref{uprc2}, which is referred to as the UASR for elastic waves, is the upward radiation condition.
The downward ASR of $u^{sc}$ in $x_3<b$ can be similarly derived and written as
\be\nonumber
 u^{\rm sc}(x)\!\!\!\!&=&\!\!\!\frac{1}{2\pi}\int_{\mathbb R^2}\Big\{
\frac{1}{\beta\,\gamma+| \xi|^2}  \Big(M^{(D)}_{\rm p}(\xi) e^{{\rm i} (\xi\cdot x^{\prime}- \beta(x_3-b))} + M^{(D)}_{\rm s}(\xi) e^{{\rm i} (\xi\cdot x^{\prime}-\gamma (x_3-b))}  \Big)   \hat{u}^{\rm sc}(\xi, b)
  \Big\} {\rm d} \xi \\ \label{dprc}
\!\!\!\! &=&\!\!\!\!\frac{1}{2\pi}\int_{\mathbb R^2}\Big\{
M^-_{\rm p}(\xi) e^{{\rm i} (\xi\cdot x^{\prime}- \beta (x_3-b))} + \Big(\textbf{\rm I}-M^-_{\rm p}(\xi)\Big) e^{{\rm i} (\xi\cdot x^{\prime}-\gamma (x_3-b))}  \Big\}\hat{u}^{\rm sc}(\xi, b)
  {\rm d} \xi.
\en
Here $M_{\rm p}^-(\xi):= M_{\rm p}^{(D)}(\xi)/(\beta \gamma+|\xi|^2)$,
\begin{equation}\label{MpsD}
M_{\rm p}^{(D)}(\xi):=\begin{bmatrix}
\xi_1^2        &  \xi_1\xi_2    & -\xi_1 \gamma\\
\xi_1\xi_2     & \xi_2^2        & -\xi_2 \gamma\\
-\xi_1\beta  & -\xi_2\beta  & \beta \gamma
\end{bmatrix},\quad
M_{\rm s}^{(D)}(\xi):=
\begin{bmatrix}
\beta\gamma+\xi_2^2 & -\xi_1\xi_2               & \gamma \xi_1\\
-\xi_1 \xi_2             & \beta\gamma+\xi_1^2  & \gamma \xi_2\\
\xi_1\beta           & \xi_2\beta            & |\xi|^2
\end{bmatrix}.
\end{equation}

%\begin{rema}
If $u^{\rm sc}$ is quasi-biperiodic on $\Gamma_b$, then the ASR of $u^{\rm sc}$ in a half space is equivalent to the Rayleigh expansion of $u^{\rm sc}$ (see \cite{Aren99,eh-mmmas12,eh-2010}).
We say $u^{\rm sc}$ is quasi-biperiodic
 with the phase-shift $\alpha=(\alpha_1,\alpha_2)\in \R^2$ in the variable $x^{\prime}$, if $u^{\rm sc}(x^{\prime}+2\pi n,b) = e^{{\rm i}2\pi\alpha\cdot n}u^{\rm sc}(x^{\prime},b)$ for all $n=(n_1,n_2)\in \mathbb Z^2$. Therefore, $u^{\rm sc}(x',b)$ admits the Fourier series expansion
\begin{align}\label{fe}
u^{\rm sc}(x^{\prime},b) = \sum_{n\in \mathbb Z^2}u^{\rm sc}_n(b) e^{{\rm i}\alpha_n\cdot x^{\prime}},\quad x'\in \R^2,
\end{align}
where $\alpha_n=\alpha + n$ and $u^{\rm sc}_n(b)$ is the Fourier coefficient of $u^{\rm sc}$ on $\Gamma_b$, given by
\[
u^{\rm sc}_n(b) = \frac{1}{2\pi}\int_0^{2\pi}\int_0^{2\pi}u^{\rm sc}(x^{\prime},b)e^{-{\rm i}\alpha_n\cdot x^{\prime}}{\rm d}x^{\prime}.
\]
Substituting \eqref{fe} into \eqref{uprc} and noting that the Fourier transform of $e^{{\rm i}\alpha_n\cdot x^{\prime}}$ is $2\pi \delta(\xi - \alpha_n)$, we obtain
\begin{align}\label{Ray}
u^{\rm sc}(x)&=\frac{1}{2\pi}\int_{\mathbb R^2}\Big\{
\frac{1}{\beta\,\gamma+| \xi|^2}  \Big(M_{\rm p}(\xi) e^{{\rm i} (\xi\cdot x^{\prime}+ \beta\, (x_3-b))} + M_{\rm s}(\xi) e^{{\rm i} (\xi\cdot x^{\prime}+\gamma\, (x_3-b))}  \Big)   \hat{ u}^{\rm sc}(\xi, b)
  \Big\} {\rm d} \xi\notag\\
  & = \sum_{n\in\mathbb Z^2}\int_{\mathbb R^2}\Big\{
\frac{1}{\beta\,\gamma+| \xi|^2}  \Big(M_{\rm p}(\xi) e^{{\rm i} (\xi\cdot x^{\prime}+ \beta\, (x_3-b))} + M_{\rm s}(\xi) e^{{\rm i} (\xi\cdot x^{\prime}+\gamma\, (x_3-b))}  \Big)  \delta(\xi - \alpha_n)
  \Big\} u^{\rm sc}_n(b){\rm d} \xi\notag\\
  & = \sum_{n\in\mathbb Z^2}
\frac{1}{\beta_n\,\gamma_n+|\alpha_n|^2}  \Big(M_{{\rm p},n} e^{{\rm i} (\alpha_n\cdot x^{\prime}+ \beta_n\, (x_3-b))} + M_{{\rm s},n} e^{{\rm i} (\alpha_n\cdot x^{\prime}+\gamma_n\, (x_3-b))}  \Big)u^{\rm sc}_n(b)\notag\\
& = \sum_{n\in\mathbb Z^2}
\frac{ (\alpha_n,\gamma_n)^\top\cdot u^{\rm sc}_n(b)}{\beta_n\,\gamma_n+|\alpha_n|^2}(\alpha_n, \beta_n)^\top e^{{\rm i} (\alpha_n\cdot x^{\prime}+ \beta_n\, (x_3-b))} \notag\\
&\qquad + \frac{1}{\beta_n\,\gamma_n+|\alpha_n|^2} \Big[(\alpha_n, \gamma_n)^\top\times \Big(u^{\rm sc}_n(b)\times(\alpha_n, \beta_n)^\top\Big)\Big] e^{{\rm i} (\alpha_n\cdot x^{\prime}+\gamma_n\, (x_3-b))},
\end{align}
where
\begin{align*}
\beta_n = \beta(\alpha_n), \quad \gamma_n = \gamma(\alpha_n), \quad M_{{\rm p},n}=M_{\rm p}(\alpha_n), \quad M_{{\rm s},n}=M_{\rm s}(\alpha_n).
\end{align*}
The representation \eqref{Ray} is the upward Rayleigh expansion of $u^{\rm sc}$ in $x_3>b$.
Using the vector identity
\ben
&&(\alpha_n, \gamma_n)^\top\times \Big(u^{\rm sc}_n(b)\times(\alpha_n, \beta_n)^\top\Big)\\
&=&
\Big((\alpha_n, \gamma_n)\cdot (\alpha_n, \beta_n)^\top\Big)
u^{\rm sc}_n(b)
-\Big((\alpha_n, \gamma_n)^\top\cdot u^{\rm sc}_n(b)\Big)(\alpha_n, \beta_n)^\top \\
&=& (\beta_n \gamma_n+|\alpha_n|^2)\,u^{\rm sc}_n(b)
-\Big((\alpha_n, \gamma_n)^\top\cdot u^{\rm sc}_n(b)\Big)(\alpha_n, \beta_n)^\top,
\enn
we may rewrite \eqref{Ray} into
\be\label{Rayleigh}
u^{sc}(x)=\sum_{n\in\mathbb Z^2}
A_{{\rm p},n}(\alpha_n, \beta_n)^\top e^{{\rm i} (\alpha_n\cdot x^{\prime}+ \beta_n (x_3-b))}
 +\textbf{A}_{{\rm s},n} e^{{\rm i} (\alpha_n\cdot x^{\prime}+\gamma_n (x_3-b))},
\en
where
\ben
A_{{\rm p},n}=\frac{ (\alpha_n,\gamma_n)^\top\cdot u^{\rm sc}_n(b)}{\beta_n\gamma_n+|\alpha_n|^2}
\in \C,\quad \textbf{A}_{{\rm s},n}=u^{\rm sc}_n(b)-A_{{\rm p},n}(\alpha_n, \beta_n)^\top\in \C^3.
\enn

Clearly, it holds that $(\alpha_n, \gamma_n)\cdot \textbf{A}_{{\rm s},n}=0$ for all $n\in \mathbb Z^2$.
The representation (\ref{Rayleigh}) is the reduction of the UASR (see (\ref{uprc}) and (\ref{uprc1})) to the Rayleigh expansion in quasi-periodic spaces. The equivalence of the downward radiation conditions can be justified in the same manner.
%\end{rema}

The rough surface scattering problem can be stated as follows: Given a plane incident wave \eqref{inc} or a point incident wave \eqref{gtn}, the scattering problem is to find the scattered field $u^{\rm sc}$ of the boundary value problem for the Navier equation \eqref{tf} in a distributional sense, such that the upward radiation condition \eqref{uprc} is satisfied.

The goal of this paper is twofold:
 \begin{enumerate}
 \item Prove uniqueness of the solution in $H^1(S_b)^3$ for any $b>f^+$ (see Section \ref{uniqueness});

 \item For locally perturbed flat surfaces, prove existence of the Kupradze radiating solution $u^{\rm sc}-u^{\rm re}\in H^1_{\rm loc}(D)^3$, where $u^{\rm re}$ denotes the reflected wave field corresponding to the unperturbed flat surface (see Section \ref{existence}).
 \end{enumerate}

In the subsequent section, we will introduce a Dirichlet-to-Neumann (DtN) map on the artificial flat surface $\Gamma_b$ for some $b>f^+$ and investigate its mapping properties.

\section{Dirichlet-to-Neumann map}\label{sec:3}
\setcounter{equation}{0}

Recall that the traction operator on a surface is defined as
\begin{align*}
Tu:=2\mu\partial_{\nu}u + \lambda (\nabla\cdot u)\nu + \mu \nu\times (\nabla\times u),
\end{align*}
where $\nu=(\nu_1,\nu_2,\nu_3)$ stands for the normal vector on the surface. Given $b>f^+$, the DtN map for our rough surface scattering problem is defined as follows.

\begin{defi}
For $v\in H^{1/2}(\Gamma_b)^3$, the upward DtN map $\mathcal{T}v$ is defined as $T u^{\rm sc}$ on $\Gamma_b$, where $u^{\rm sc}$ is the unique upward radiation solution of the homogeneous Navier equation in $U_b$  satisfying  $u^{\rm sc} = v$ on $\Gamma_b$. More explicitly, we have
\begin{align}\label{traction}
Tu:=2\mu\partial_{3}u + \lambda (\nabla\cdot u)(0,0,1)^{\top} + \mu (0,0,1)^{\top}\times (\nabla\times u).
\end{align}
\end{defi}

Note that the above definition is well defined, because $u^{\rm sc}$ can be uniquely determined in $U_b$ via the formula \eqref{uprc}.
Next we derive an explicit representation of the upward DtN map $\mathcal{T}$ and show some of its properties.

Applying the traction operator $T$ \eqref{traction} to \eqref{uprc} and letting $x_3 = b$, we get
\be\label{Fv}
\mathcal{F}[(Tu^{\rm sc})|_{\Gamma_b}](\xi)={\rm i}
\begin{bmatrix}
2\mu\beta\xi_1 & \mu\gamma & 0 & \mu\xi_1\\[5pt]
2\mu\beta\xi_2 & 0 & \mu\gamma & \mu\xi_2\\[5pt]
2\mu\beta^2 + \lambda \kappa_{\rm p}^2 &  0 & 0 & 2\mu\gamma
\end{bmatrix}
\begin{bmatrix}
A_{\rm p}\\[5pt]
\boldsymbol A_{\rm s}^{\top}
\end{bmatrix}=:{\rm i}G(\xi)\boldsymbol A(\xi).
\en
Recalling  $\boldsymbol A(\xi)=\mathbb{D}(\xi)\hat{u}^{\rm sc}(\xi,b)$ in (\ref{A}), we have
\be\label{M}
\mathcal{F}[(Tu^{\rm sc})|_{\Gamma_b}](\xi)={\rm i}G(\xi)\mathbb{D}(\xi) \hat{u}^{\rm sc}(\xi,b)={\rm i}M(\xi)\hat{u}^{\rm sc}(\xi,b),
\en
where $M(\xi)=G(\xi) \mathbb{D}(\xi)\in \C^{3\times 3}$ is given by
\begin{align*}
M(\xi) = \frac{1}{|\xi|^2+\beta\gamma}
\begin{bmatrix}
\mu[(\gamma - \beta)\xi_2^2 + \kappa_{\rm s}^2\beta] & -\mu\xi_1\xi_2(\gamma - \beta) & (2\mu| \xi|^2 - \omega^2 + 2\mu\beta \gamma)\xi_1 \\[5pt]
-\mu\xi_1\xi_2(\gamma - \beta) & \mu[(\gamma - \beta)\xi_1^2 + \kappa_{\rm s}^2\beta] & (2\mu| \xi|^2 - \omega^2 + 2\mu\beta\gamma)\xi_2\\[5pt]
- (2\mu| \xi|^2 - \omega^2 + 2\mu\beta \gamma)\xi_1 & -(2\mu| \xi|^2 - \omega^2 + 2\mu\beta\gamma)\xi_2 & \gamma \omega^2
\end{bmatrix}.
\end{align*}
Taking the inverse Fourier transform gives
\ben
[\mathcal{T} {u}^{\rm sc}(x',b)](x')=\frac{{\rm i}}{2\pi}
\int_{\mathbb R^2}G(\xi)\boldsymbol A(\xi)e^{{\rm i}  \xi \cdot x^{\prime}}{\rm d} \xi=\frac{{\rm i}}{2\pi}
\int_{\mathbb R^2}M( \xi)
\hat{u}^{\rm sc} (\xi, b) e^{{\rm i}  \xi \cdot x^{\prime}}{\rm d} \xi,
\enn
where the matrix function $M$ is given in \eqref{M}. Since $v=u^{\rm sc}|_{\Gamma_b}$, we obtain the upward DtN map
\be\label{DtN}
\mathcal{T} v(x')=\frac{{\rm i}}{2\pi}
\int_{\mathbb R^2}M( \xi)
\hat{v}(\xi) e^{{\rm i}  \xi \cdot x^{\prime}}{\rm d} \xi.
\en
The boundary operator $\mathcal{T}$ is non-local and is equivalent to the upward radiation condition (\ref{uprc}).

Similarly, we may show that the downward DtN map takes the form
\be\label{Mn}
\mathcal{T}^- v (x')=\frac{{\rm i}}{2\pi} \int_{\mathbb R^2}M^-(\xi)
\hat{v}(\xi) e^{{\rm i}  \xi \cdot x^{\prime}}{\rm d} \xi,
\en
with
\ben
M^-(\xi) = \frac{1}{|\xi|^2+\beta\gamma}
\begin{bmatrix}
-\mu[(\gamma - \beta)\xi_2^2 + \kappa_{\rm s}^2\beta] & \mu\xi_1\xi_2(\gamma - \beta) & (2\mu| \xi|^2 - \omega^2 + 2\mu\beta \gamma)\xi_1 \\[5pt]
\mu\xi_1\xi_2(\gamma - \beta) & -\mu[(\gamma - \beta)\xi_1^2 + \kappa_{\rm s}^2\beta] & (2\mu| \xi|^2 - \omega^2 + 2\mu\beta\gamma)\xi_2\\[5pt]
- (2\mu| \xi|^2 - \omega^2 + 2\mu\beta \gamma)\xi_1 & -(2\mu| \xi|^2 - \omega^2 + 2\mu\beta\gamma)\xi_2 & \gamma \omega^2
\end{bmatrix}.
\enn

In comparision with the matrix $M$ for the upward DtN (cf. (\ref{M})), the parameters $\beta(\xi),\gamma(\xi)$ are  replaced by
$-\beta(\xi), -\gamma(\xi)$ in the definition of $M^-(\xi)$, respectively.

\begin{lemm}\label{dtn}
Let $M(\xi)$ be defined in \eqref{M} and let $b>f^+$.
\begin{enumerate}
\item Given a fixed frequency $\omega>0$, we have $ \Re(-{\rm i}M)(\xi)>0$ for all sufficiently large $|\xi|$.
\item The DtN map $\mathcal{T}$ is a bounded operator from $H^{1/2}(\Gamma_b)^3$ to $H^{-1/2}(\Gamma_b)^3$.
 %\item Let $u^{sc}\in H^1(S_b)$ be an upward radiation solution of the form \eqref{uprc} or \eqref{uprc1}. Then, for $v:=u^{sc}(\cdot,b)$ it holds that
 %\begin{align}\label{4.3g}
%{\Im} \int_{\Gamma_b}  \overline{v} \cdot \mathcal{T} v {\rm d}x^{\prime}
%= \int_{| \xi|<\kappa_{\rm p}} \omega^2\beta( \xi)|A_{\rm p}( \xi)|^2\omega^2 {\rm d} \xi + \int_{| \xi|<\kappa_{\rm s}}\mu\gamma( \xi)|\boldsymbol A_{\rm s}( \xi)|^2 \mu {\rm d} \xi.
%\end{align}
\end{enumerate}
\end{lemm}

The proof of Lemma \ref{dtn} relies essentially on properties of the matrix $M$ and can be carried out following almost the same arguments in the quasi-periodic case of \cite{eh-mmmas12}. We omit the details for brevity.

\section{Uniqueness}\label{cuniqueness}
\setcounter{equation}{0}

In this section, we study the uniqueness for our boundary value problem if $\Gamma$ is the graph of a uniformly Lipschitz continuous function $f$, i.e.,
\[
\Gamma = \{x\in \mathbb R^3: x_3 = f(x^{\prime}), \, x^{\prime}=(x_1, x_2) \in \mathbb R^2\},
\]
and there exists a constant $L>0$ such that
\[
|f(x^{\prime}) - f(y^{\prime})|\leq L\,|x^{\prime} - y^{\prime}|,  \quad \forall x^{\prime}, y^{\prime} \in \mathbb R^2.
\]

First, we investigate the uniqueness when $f$ is a $C^2$-smooth function over $\mathbb R^2$. Denote the unit normal vector on $\Gamma \cup \Gamma_b$ by $\nu: = (\nu_1, \nu_2, \nu_3)$ pointing into the region of $x_3>b$ on $\Gamma_b$ and into the interior of $D$ on $\Gamma$.
In the rest of this subsection, we assume that $u^{\rm in}=0$ and thus $u=u^{\rm sc}$ is a radiation solution in $S_b$ for any $b>f^+$. We shall prove that $u\equiv 0$ in $D$, depending on the geometry of $\partial D$. This result implies that elastic surface waves are ruled out if the rigid scattering surface is given by some uniformly Lipschitz function.
Our uniqueness proof depends on a Rellich-type identity for the Navier equation in the unbounded strip $S_b$. The Rellich-type identity was first used in \cite{Kirsch93} to prove uniqueness of the acoustic scattering by smooth periodic sound-soft curves and in \cite{J.M2002} for treating periodic Lipschitz graphs. Besides, it gave a priori estimates and explicit bounds on the solution of the acoustic rough surface scattering problems \cite{cm-siaa05}. We refer to \cite{Peter} for more general Rellich's identities in a bounded domain.

 \begin{lemm}\label{rellich}
If $u \in H^1(S_b)^3$ and $f$ is a $C^2$-smooth function, the following Rellich identity holds:
\begin{align*}
&2\Re \int_{S_b} (\mu \Delta u + (\lambda + \mu)\nabla \cdot u + \omega^2 u)\cdot \partial_3 \bar{u}{\rm d}x\\
&= \Big(-\int_{\Gamma} + \int_{\Gamma_b}\Big) \Big\{2\Re(Tu \cdot \partial_3 \bar{u}) - \nu_3 \mathcal{E}(u,\bar{u}) + \omega^2|u|^2 \Big\}{\rm d}s.
\end{align*}
\end{lemm}

\begin{proof}
The proof is similar as that in \cite[Lemma 6]{eh-siaa12}. We sketch it here. By standard elliptic regularity, we see that $u \in H^2(S_b)^3$. For $A \geq 1$, we choose a cut-off function $\chi_A (r) \in C_0^{\infty}(\mathbb R^+)$ with $r = |x|$ such that $\chi_A(r) = 1$ if $r \leq A$, $\chi_A(r) = 0$ if $r \geq A + 1$, $0\leq \chi_A(r)\leq 1$ if $A < r \leq A + 1$, and $\|\chi'_A(r)\|\leq C$ for some fixed $C$ independent of $A$. Multiplying both sides of \eqref{tf} by the test function $\chi_A(r)\partial_3 \bar{u}$, using the integration by parts, and letting $A \rightarrow +\infty$, we may obtain the desired identity.
\end{proof}

Since $u$ satisfies the Navier equation in $D$, it follows from Lemma \ref{rellich} that
\be\nonumber
\int_{\Gamma}  \Big\{2\Re(Tu \cdot \partial_3 \bar{u}) - \nu_3 \mathcal{E}(u,\bar{u}) + \omega^2|u|^2 \Big\}{\rm d}s%\\ \label{eq:1}
=\int_{\Gamma_b} \Big\{2\Re(Tu \cdot \partial_3 \bar{u}) - \nu_3 \mathcal{E}(u,\bar{u}) + \omega^2|u|^2 \Big\}{\rm d}s.
\en
In the following lemma, we simply the left hand side of the above identity by using the boundary condition $u=0$ on $\Gamma$ and simply the right hand side of the above identity by the radiation condition of $u=u^{\rm sc}$.

\begin{lemm}\label{lem:3.6}
(i) Under the assumptions of Lemma \ref{rellich}, it holds that
\ben
\int_{\Gamma}  \Big\{2\Re(Tu \cdot \partial_3 \bar{u}) - \nu_3 \mathcal{E}(u,\bar{u}) + \omega^2|u|^2 \Big\}{\rm d}s
= \int_{\Gamma} \mu|\partial_{\nu}u|^2\nu_3 + (\lambda + \mu) |\nabla \cdot u|^2 \nu_3 {\rm d}s.
\enn
(ii) Let $u=u^{\rm sc}$ satisfy \eqref{uprc} in $x_3>b$ with the parameter-dependent coefficients $A_{\rm p}(\xi)$ and $\textbf{A}_{\rm s}(\xi)\in \C^{3\times 1}$ for $\xi\in \R^3$. We have
\be\nonumber
&&\int_{\Gamma_b} \Big\{2\Re(Tu \cdot \partial_3 \bar{u}) - \nu_3 \mathcal{E}(u,\bar{u}) + \omega^2|u|^2 \Big\}{\rm d}s\\ \label{eq:4}
&&\;= 2\omega^2 \Big\{ \int_{|\xi|<\kappa_{\rm p}} \beta^2(\xi) |A_{\rm p}(\xi)|^2\,{\rm d}\xi+
\int_{|\xi|<\kappa_{\rm s}} \gamma^2(\xi) |\textbf{A}_{\rm s}(\xi)|^2\,{\rm d}\xi
\Big\},\\ \label{eq:a}
&&\Im\int_{\Gamma_b} Tu \cdot \bar{u}{\rm d}s =
\int_{|\xi|<\kappa_{\rm p}} \omega^2\beta( \xi)|A_{\rm p}(\xi)|^2 {\rm d} \xi + \int_{|\xi|<\kappa_{\rm s}}\mu\gamma(\xi)|\boldsymbol A_{\rm s}( \xi)|^2{\rm d} \xi.
\en
\end{lemm}

\begin{proof}
(i) Since $u=0$ on $\Gamma$, a direct calculation shows that on $\Gamma$ (see also \cite[Lemma 5]{eh-mmmas12}),
\ben
\nu \cdot \partial_3 \bar{u} \nabla \cdot u = \nu_3 |\nabla \cdot u|^2,\quad \partial_3 u = \nu_3 \partial_{\nu}u,\quad
 \partial_{\nu} u + \nu \times (\nabla \times u) - \nu \nabla \cdot u = 0.
\enn Hence, by the definitions of the traction operator $T$ and the bilinear form $\mathcal{E}(\cdot,\cdot)$, we get
\ben
 Tu \cdot \partial_3 \bar{u} = \mathcal{E}(u,\bar{u})=\nu_3 \mu|\partial_{\nu}u|^2\nu_3 + (\lambda + \mu) |\nabla \cdot u|^2 \nu_3,
\enn
which proves the first assertion.

(ii) The proof of the second  assertion depends on the upward ASR of $u=u^{\rm sc}$ and the Parseval formula.

It follows from \eqref{Fv} and the Fourier transform of $Tu$ in terms of $A_{\rm p}$ and $\boldsymbol A_{\rm s}$ on $\Gamma_b$ that
$\widehat{Tu}(\xi) ={\rm i} G(\xi) \boldsymbol A(\xi)$, where $\boldsymbol A$ is defined in \eqref{A}.
By \eqref{uprc}, the Fourier transform $\widehat{\partial_j u}$ of $\partial_j u$ on $\Gamma_b$  can be represented by
\begin{align*}
\widehat{\partial_j u} = H_j(\xi)\,\boldsymbol A(\xi),\quad j=1,2,3,
\end{align*}
where $H_j$ are 3-by-4 matrices defined by
\begin{align*}
H_1={\rm i}\begin{bmatrix}
\xi_1^2 & \xi_1 & 0 & 0\\[5pt]
\xi_1 \xi_2 & 0 & \xi_1 & 0\\[5pt]
\xi_1 \beta & 0 & 0 & \xi_1
\end{bmatrix},\;
H_2={\rm i}\begin{bmatrix}
\xi_1 \xi_2& \xi_2 & 0 & 0\\[5pt]
\xi_2^2 & 0 & \xi_2 & 0\\[5pt]
\xi_2 \beta & 0 & 0 & \xi_2
\end{bmatrix},\;
H_3 = {\rm i}\begin{bmatrix}
\beta \xi_1 & \gamma & 0 & 0\\[5pt]
\beta \xi_2 & 0 & \gamma & 0\\[5pt]
\beta^2 & 0 & 0 & \gamma
\end{bmatrix}.
\end{align*}
Direct calculations show that
\begin{align*}
&H_1^*H_1 = \begin{bmatrix}
\xi_1^2(|\xi|^2 + |\beta|^2)& \xi_1^3 & \xi_1^2 \xi_2 & \xi_1^2 \bar{\beta}\\[5pt]
\xi_1^3 & \xi_1^2 & 0 & 0\\[5pt]
\xi_1^2\xi_2 & 0 & \xi_1^2 & 0\\[5pt]
\xi_1^2\beta & 0 & 0 & \xi_1^2
\end{bmatrix},
H_2^*H_2 = \begin{bmatrix}
\xi_2^2(|\xi|^2 + |\beta|^2)& \xi_1\xi_2^2 &  \xi_2^3 & \xi_2^2 \bar{\beta}\\[5pt]
\xi_1\xi_2^2 & \xi_2^2 & 0 & 0\\[5pt]
\xi_2^3 & 0 & \xi_2^2 & 0\\[5pt]
\xi_2^2\beta & 0 & 0 & \xi_2^2
\end{bmatrix},\\
&H_3^*H_3 = \begin{bmatrix}
|\beta|^2(|\xi|^2 + |\beta|^2) & \gamma\bar{\beta}\xi_1 & \gamma\bar{\beta}\xi_2 & \gamma \bar{\beta}^2\\[5pt]
\bar{\gamma}\beta \xi_1 & |\gamma|^2 & \gamma & 0\\[5pt]
\bar{\gamma}\beta \xi_2 & 0 & |\gamma|^2 & 0\\[5pt]
\bar{\gamma}\beta^2 & 0 & 0 & |\gamma|^2
\end{bmatrix}.
\end{align*}
Moreover we have
\begin{align}\label{M1}
M_1:= H_1^* G = \begin{bmatrix}
2\mu |\beta|^2 (|\xi|^2 + |\beta|^2) + \lambda \kappa_{\rm p}^2\bar{\beta}^2 & \mu\bar{\beta}\xi_1\gamma & \mu\bar{\beta}\xi_2\gamma & \mu\bar{\beta}|\xi|^2 + 2\mu\bar{\beta}^2\gamma\\[5pt]
2\mu\beta\bar{\gamma}\xi_1 & \mu|\gamma|^2 & 0 & \mu\xi_1\bar{\gamma}\\[5pt]
2\mu\beta\bar{\gamma}\xi_2 & 0 & \mu|\gamma|^2 & \mu\xi_2\bar{\gamma} \\[5pt]
2\mu\beta^2\bar{\gamma} + \lambda \kappa_{\rm p}^2 \bar{\gamma} & 0 & 0 & 2\mu|\gamma|^2
\end{bmatrix}
\end{align}
and
\begin{align*}
M_2&:= H_1^*H_1 + H_2^*H_2 + H_3^*H_3 \\
&= \begin{bmatrix}
(|\xi|^2 + |\beta|^2)^2& \xi_1(|\xi|^2 + \gamma \bar{\beta}) & \xi_2(|\xi|^2 + \gamma \bar{\beta}) & \bar{\beta}(|\xi|^2 + \gamma \bar{\beta})\\[5pt]
\xi_1(|\xi|^2 + \beta \bar{\gamma}) & |\xi|^2 + |\gamma|^2 & 0 & 0\\[5pt]
\xi_2(|\xi|^2 + \beta \bar{\gamma}) & 0 & |\xi|^2 + |\gamma|^2 & 0\\[5pt]
\beta(|\xi|^2 + \beta \bar{\gamma}) & 0 & 0 & |\xi|^2 + |\gamma|^2
\end{bmatrix}.
\end{align*}

The Fourier transforms of $u$, $\nabla\cdot u$ and $\nabla\times u$ on $\Gamma_b$ are given respectively by
\begin{align*}
\hat{u}( \xi, b)=\mathbb{D}_1(\xi) \boldsymbol A(\xi),\quad
\widehat{\nabla \cdot  u} = H_4(\xi) \boldsymbol A(\xi),\quad
\widehat{\nabla \times u}=(\xi, \gamma)^\top\times \boldsymbol A_{\rm s}(\xi),
\end{align*}
where
\ben
\mathbb{D}_1(\xi)=\begin{bmatrix}
\xi_1 & 1 & 0 & 0\\
\xi_2 & 0 & 1 & 0\\
\beta & 0 & 0 & 1
\end{bmatrix},\quad
H_4 = {\rm i}\begin{bmatrix}
\kappa_{\rm p}^2 & \xi_1\xi_2^2 &  \xi_2^3 & \xi_2^2 \bar{\beta}\\[5pt]
\xi_1\xi_2^2 & \xi_2^2 & 0 & 0\\[5pt]
\xi_2^3 & 0 & \xi_2^2 & 0\\[5pt]
\xi_2^2\beta & 0 & 0 & \xi_2^2
\end{bmatrix}.
\enn
Simple calculations yield
\begin{align*}
M_3:= H_4^*H_4 = \begin{bmatrix}
\kappa_{\rm p}^4 & 0 &  0 & 0\\[5pt]
0 & 0& 0 & 0\\[5pt]
0 & 0 & 0 & 0\\[5pt]
0 & 0 & 0 & 0
\end{bmatrix},\quad
M_4:= \mathbb{D}_1^*\mathbb{D}_1 =
\begin{bmatrix}
|\xi|^2 + |\beta|^2 & \xi_1 & \xi_2 & \bar{\beta}\\[5pt]
\xi_1 & 1 & 0 & 0 \\[5pt]
\xi_2 & 0 & 1 & 0 \\[5pt]
\beta & 0 & 0 & 1
\end{bmatrix},
\end{align*}
and $|\widehat{\nabla \times u}|^2 = (|\xi|^2 + |\gamma|^2)|\boldsymbol A_{\rm s}|^2$ due to the orthogonal identity
 $(\xi, \gamma) \cdot \boldsymbol A_{\rm s}=0$. Denote
\begin{align*}
M_5:= \begin{bmatrix}
0 & 0 & 0 & 0\\[5pt]
0 & |\xi|^2 + |\gamma|^2 & 0 & 0 \\[5pt]
0 & 0 & |\xi|^2 + |\gamma|^2 & 0 \\[5pt]
0 & 0 & 0 & |\xi|^2 + |\gamma|^2
\end{bmatrix}.
\end{align*}

By the definition of $M_j, j=1,2\cdots, 5$ and the Parseval formula, we obtain
\ben
\int_{\Gamma_b}Tu \cdot \partial_3 \bar{u}{\rm d}s&=&\int_{\R^2}  M_1(\xi) \boldsymbol A(\xi)\cdot\overline  {\boldsymbol A}(\xi)\;{\rm d}\xi\,,\\
\int_{\Gamma_b}\mathcal{E}(u,\overline{u}){\rm d}s&=&\int_{\R^2} \Big( 2\mu M_2(\xi)+\lambda M_3(\xi)-\mu M_5(\xi)\Big) \boldsymbol A(\xi)\cdot\overline{  \boldsymbol A}(\xi)\;{\rm d}\xi\,,\\
\int_{\Gamma_b} |u|^2\,{\rm d}s&=&\int_{\R^2} M_4(\xi)\boldsymbol A(\xi)\cdot\overline{\boldsymbol A}(\xi)\;{\rm d}\xi.
\enn
Hence,
\begin{align}\label{W}
\int_{\Gamma_0} 2\Re (Tu \cdot \partial_3 \bar{u}) - \mathcal{E}(u, \bar{u}) + \omega^2 |u|^2 {\rm d}s
= \int_{\mathbb R^2}  [\Re W(\xi)] \boldsymbol A(\xi)\cdot\overline{ \boldsymbol A}(\xi) {\rm d}\xi,
\end{align}
where
\[
W := 2M_1 - 2\mu M_2 - \lambda M_3 + \mu M_5 + \omega^2 M_4.
\]

Next we need to calculate $\Re W$.
To obtain the real part of $M_1$, we decompose it into the sum $J_{1,1} + J_{1,2} + J_{1,3}$, where (e.g., (\ref{M1}))
\begin{align*}
J_{1,1} &= \begin{bmatrix}
2\mu |\beta|^2 (|\xi|^2 + |\beta|^2) + \lambda \kappa_{\rm p}^2\bar{\beta}^2 & 0 & 0 & 0\\[5pt]
0 & \mu|\gamma|^2 & 0 & 0\\[5pt]
0 & 0 & \mu|\gamma|^2 & 0\\[5pt]
0 & 0 & 0 & \mu|\gamma|^2
\end{bmatrix},\\
J_{1,2} &= \begin{bmatrix}
0 & \mu\bar{\beta}\xi_1\gamma & \mu\bar{\beta}\xi_2\gamma & \mu\bar{\beta}|\xi|^2 + 2\mu\bar{\beta}^2\gamma\\[5pt]
2\mu\beta\bar{\gamma}\xi_1 & 0 & 0 & 0\\[5pt]
2\mu\beta\bar{\gamma}\xi_2 & 0 & 0 & 0 \\[5pt]
2\mu\beta^2\bar{\gamma} + \lambda \kappa_{\rm p}^2 \bar{\gamma} & 0 & 0 & 0
\end{bmatrix},\\
J_{1,3} &= \begin{bmatrix}
0 & 0 & 0 &0\\[5pt]
0 & 0 & 0 & \mu\xi_1\bar{\gamma}\\[5pt]
0 & 0 & 0 & \mu\xi_2\bar{\gamma} \\[5pt]
0 & 0 & 0 & \mu|\gamma|^2
\end{bmatrix},\; \tilde{J}_{1,2} = \begin{bmatrix}
0 & 0 & 0 & (2\mu(|\xi|^2 + \bar{\beta}\gamma) - \omega^2) \overline{\beta}\\[5pt]
0 & 0 & 0 & 0\\[5pt]
0 & 0 & 0 & 0 \\[5pt]
-(2\mu(|\xi|^2 +\beta\bar{\gamma}) - \omega^2) \overline{\gamma} & 0 & 0 & 0
\end{bmatrix}.
\end{align*}
Using the relations
\[
\xi_1 A_{\rm s}^{(1)}+ \xi_2 A_{\rm s}^{(2)} + \gamma A_{\rm s}^{(3)} = 0, \quad |\xi|^2 + \beta^2 = \kappa_{\rm p}^2, \quad |\xi|^2 + \gamma^2 = \kappa_{\rm s}^2
\]
we obtain
\begin{align}\label{12}
\langle J_{1,2} \boldsymbol A, \boldsymbol A \rangle = \langle \tilde{J}_{1,2} \boldsymbol A, \boldsymbol A \rangle, \quad \langle J_{3,2} \boldsymbol A, \boldsymbol A \rangle = 0.
\end{align}

Similarly, we decompose $M_2$ into the sum $J_{2,1} + J_{2,2}$, where
\begin{align*}
J_{2,1} &= \begin{bmatrix}
(|\xi|^2 + |\beta|^2)^2& 0 & 0 & 0\\[5pt]
0 & |\xi|^2 + |\gamma|^2 & 0 & 0\\[5pt]
0 & 0 & |\xi|^2 + |\gamma|^2 & 0\\[5pt]
0 & 0 & 0 & |\xi|^2 + |\gamma|^2
\end{bmatrix},\\
J_{2,2} &= \begin{bmatrix}
0& \xi_1(|\xi|^2 + \gamma \bar{\beta}) & \xi_2(|\xi|^2 + \gamma \bar{\beta}) & \bar{\beta}(|\xi|^2 + \gamma \bar{\beta})\\[5pt]
\xi_1(|\xi|^2 + \beta \bar{\gamma}) & 0 & 0 & 0\\[5pt]
\xi_2(|\xi|^2 + \beta \bar{\gamma}) & 0 & 0 & 0\\[5pt]
\beta(|\xi|^2 + \beta \bar{\gamma}) & 0 & 0 & 0
\end{bmatrix}.\\
\end{align*}
A simple calculation yields
\begin{align}\label{2}
\langle J_{2,2} \boldsymbol A, \boldsymbol A \rangle = \langle \tilde{J}_{2,2} \boldsymbol A\rangle,
\end{align}
where $\tilde{J}_{2,2}$ is the $4 \times 4$ matrix whose $(1,4)$-th entry is $-2\mu(|\xi|^2 + \gamma \bar{\beta})(\bar{\beta} - \gamma)$ and $(4,1)$-th entry is $-2\mu(|\xi|^2 + \bar{\gamma} \beta)(\beta - \bar{\gamma})$. We decompose $M_4$ into the sum $J_{4,1} + J_{4,2}$, where
\begin{align*}
J_{4,1} &= \begin{bmatrix}
|\xi|^2 + |\beta|^2 & 0 & 0 & 0\\[5pt]
0 & 1 &  & 0 \\[5pt]
0 & 0 & 1 & 0 \\[5pt]
0 & 0 & 0 & 1
\end{bmatrix},\quad
J_{4,2} = \begin{bmatrix}
0 & \xi_1 & \xi_2 & \bar{\beta}\\[5pt]
\xi_1 & 0 & 0 & 0 \\[5pt]
\xi_2 & 0 & 0 & 0 \\[5pt]
\beta & 0 & 0 & 0
\end{bmatrix}.
\end{align*}
Then we obtain
\begin{align}\label{3}
\langle J_{4,2} \boldsymbol A, \boldsymbol A \rangle = \langle \tilde{J}_{4,2} \boldsymbol A\rangle,
\end{align}
where $\tilde{J}_{4,2}$ is the $4 \times 4$ matrix whose $(1,4)$-th entry is $\bar{\beta} - \gamma$ and $(4,1)$-th entry is $\beta - \bar{\gamma}$.

Combining \eqref{1} and \eqref{2}--\eqref{3}, we deduce from \eqref{W} that
\begin{align*}
\Big\langle \Re W(\xi) \boldsymbol A, \boldsymbol A \Big\rangle &= \Big\langle Q(\xi) \boldsymbol A, \boldsymbol A\Big\rangle+ \Big\langle
\Re\Big(2\tilde{J}_{1,2}-2\mu \tilde{J}_{2,2} + \omega^2 \tilde{J}_{4,2}\Big) \boldsymbol A, \boldsymbol A\Big\rangle.
\end{align*}
with $Q=(Q_{i,j})_{i,j=1}^4:=\Re\Big(2J_{1,1}-2\mu J_{2,1}-\lambda M_3 + \mu M_5 + \omega^2 J_{4,1}\Big)$.
Moreover, we can obtain $\Re\Big(2\tilde{J}_{1,2}-2\mu \tilde{J}_{2,2} + \omega^2 \tilde{J}_{4,2}\Big) = 0$,
$Q_{i,j}=0$ if $i\neq j$ and
\begin{align*}
Q_{1,1} =
\begin{cases}
2\omega^2 \beta^2, \quad &|\xi|< \kappa_{\rm p},\\
0, \quad &|\xi| > \kappa_{\rm p},
\end{cases}\quad
Q_{i, i} =
\begin{cases}
2\omega^2 \gamma^2, \quad &|\xi|< \kappa_{\rm s},\\
0, \quad &|\xi| > \kappa_{\rm s},
\end{cases}\quad\mbox{if}\quad i=2,3,4.
\end{align*}
Hence,
\begin{align*}
\int_{\mathbb R^2}\Big\langle \Re W(\xi) \boldsymbol A, \boldsymbol A \Big\rangle{\rm d}\xi &= \int_{\mathbb R^2}\Big\langle Q \boldsymbol A, \boldsymbol A\Big\rangle {\rm d}\xi\\
&= 2\omega^2 \Big( \int_{|\xi|<\kappa_1} \beta^2(\xi)|A_{\rm p}(\xi)|^2 {\rm d}\xi  + \int_{|\xi|<\kappa_2} \gamma^2(\xi)|\boldsymbol A_{\rm s}(\xi)|^2 {\rm d}\xi \Big),
\end{align*}
which together with \eqref{W} proves the relation \eqref{eq:4}.

To prove the second identity \eqref{eq:a}, we observe that
\be\label{Tuu}
\Im\int_{\Gamma_b} Tu\, \bar{u}\,{\rm d}s=
\Im\int_{\R^2} \langle {\rm i}G \boldsymbol A, \mathbb{D}_1\boldsymbol A \rangle{\rm d}\xi=
\int_{\R^2} \langle (\Re \mathbb{D}_1^* G) \boldsymbol A, \boldsymbol A \rangle{\rm d}\xi,
\en
where
\begin{align*}
\mathbb{D}_1^*G =\begin{bmatrix}
2\mu \beta (|\xi|^2 + |\beta|^2) + \lambda \kappa_{\rm p}^2\bar{\beta} & \mu\xi_1\gamma & \mu\xi_2\gamma & \mu|\xi|^2 + 2\mu\bar{\beta}^2\gamma\\[5pt]
2\mu\beta\xi_1 & \mu\gamma & 0 & \mu\xi_1\\[5pt]
2\mu\beta\xi_2 & 0 & \mu\gamma & \mu\xi_2 \\[5pt]
2\mu\beta^2 + \lambda \kappa_{\rm p}^2  & 0 & 0 & 2\mu\gamma
\end{bmatrix}.
\end{align*}
We decompose $\mathbb{D}_1^*G$ into the sum $J_1 + J_2 + J_3$, where
\begin{align*}
J_1 &= \begin{bmatrix}
2\mu \beta (|\xi|^2 + |\beta|^2) + \lambda \kappa_{\rm p}^2\bar{\beta} & 0 & 0 & 0\\[5pt]
0 & \mu\gamma & 0 & 0\\[5pt]
0 & 0 & \mu\gamma & 0\\[5pt]
0 & 0 & 0 & \mu\gamma
\end{bmatrix},\\
J_2 &= \begin{bmatrix}
0 & \mu\xi_1\gamma & \mu\xi_2\gamma & \mu|\xi|^2 + 2\mu\bar{\beta}\gamma\\[5pt]
2\mu\beta\xi_1 & 0 & 0 & 0\\[5pt]
2\mu\beta\xi_2 & 0 & 0 & 0 \\[5pt]
2\mu\beta^2 + \lambda \kappa_{\rm p}^2  & 0 & 0 & 0
\end{bmatrix},\\
J_3 &= \begin{bmatrix}
0 & 0 & 0 &0\\[5pt]
0 & 0 & 0 & \mu\xi_1\\[5pt]
0 & 0 & 0 & \mu\xi_2 \\[5pt]
0 & 0 & 0 & \mu\gamma
\end{bmatrix}.
\end{align*}
Following a similar the proof of \eqref{12}, we have
\begin{align}\label{1}
\langle J_2 \boldsymbol A, \boldsymbol A \rangle = \langle \tilde{J}_2 \boldsymbol A, \boldsymbol A \rangle, \quad \langle J_3 \boldsymbol A, \boldsymbol A \rangle = 0,
\end{align}
where $\tilde{J}_2$ is the $4 \times 4$ matrix whose $(1,4)$-th entry is $2\mu|\xi|^2+2\mu\bar{\beta}\gamma-\omega^2$ and $(4,1)$-th entry is $-2\mu|\xi|^2-2\mu\beta\bar{\gamma}+\omega^2$, and the other entries are zeros, which imply $\Re\tilde{J}_2 = 0$.  It follows from straightforward calculation that we have
\begin{align*}
\langle \Re J_1 \boldsymbol A, \boldsymbol A\rangle =
\begin{cases}
\omega^2\beta|A_{\rm p}|^2 + \mu \gamma |\boldsymbol A_{\rm s}|^2, \quad &|\xi|<\kappa_{\rm p},\\
\mu \gamma |\boldsymbol A_{\rm s}|^2, \quad &\kappa_{\rm p} \leq |\xi| <\kappa_{\rm s},\\
0, \quad &\kappa_{\rm s}<|\xi|.
\end{cases}
\end{align*}
Following \eqref{Tuu}, we deduce
\begin{align*}
\Im\int_{\Gamma_b} Tu \bar{u}{\rm d}s =  \int_{|\xi|<\kappa_{\rm p}} \omega^2 \beta(\xi)|A_{\rm p}(\xi)|^2 {\rm d}\xi  + \int_{|\xi|<\kappa_{\rm s}} \mu\,\gamma(\xi)|\boldsymbol A_{\rm s}(\xi)|^2 {\rm d}\xi,
\end{align*}
which completes the proof.
\end{proof}

The following lemma plays an important role in the subsequent analysis.
It implies that the upward propagating modes of the compressional and shear parts must vanish, if $u^{\rm in}=0$.

\begin{lemm}\label{L3}
Assume that $u^{\rm in}=0$ and the radiating solution $u^{\rm sc}\in H^1(S_b)^3$ for any $b>f_+$, then
\begin{align*}
A_{\rm p}( \xi)=0 \quad \text{for}~ | \xi|<\kappa_{\rm p} \quad \text{and} \quad \boldsymbol A_{\rm s}( \xi)=0 \quad \text{for}~ | \xi|<\kappa_{\rm s},
\end{align*}
 where $A_{\rm p}( \xi)$ and $\boldsymbol A_{\rm s}( \xi)$ are defined in \eqref{uprc1}.
\end{lemm}

\begin{proof}
Multiplying the Navier equation in \eqref{tf} by the complex conjugate of $u^{\rm sc}$ and using Betti's formula yield
\ben
0= \int_{S_b} \mathcal{E}(u^{\rm sc}, \overline{u}^{sc}) - \omega^2  u^{\rm sc} \cdot  \overline{u}^{\rm sc}{\rm d}  x - \int_{\Gamma_b}  \overline{u}^{\rm sc} \cdot T {u^{\rm sc}} {\rm d}s,
\enn
where the bilinear form
\ben
\mathcal{E}(u,  v) := 2\mu \sum_{j,k = 1}^3 \partial_k u_j \partial_k v_j + \lambda \nabla \cdot  u \nabla \cdot v - \mu \nabla \times  u \cdot \nabla \times  v,\quad\forall u, v\in H^1(S_b)^3.
\enn
Taking the imaging part and recalling the definition of DtN operator, we obtain
\begin{align*}
0=\Im\int_{\Gamma_b}  \overline{{u}}^{\rm sc} \cdot T {u}^{\rm sc} {\rm d} s=\Im\int_{\Gamma_b}  \overline{{u}}^{\rm sc} \cdot \mathcal T {u}^{\rm sc} {\rm d}s=0,
\end{align*}
which proves the result by noting \eqref{eq:a} with $u=u^{\rm sc}$.
\end{proof}

As seen from Lemma \ref{L3}, the uniqueness does not hold for general rough surfaces. In the following theorem, we investigate the uniqueness under an additional geometrical assumption of the scattering surface.

\begin{theo}\label{uniqueness}
Suppose that $\Gamma$ is the graph of a uniformly Lipschitz function and that $u^{\rm in}=0$. Then $u\equiv0$ in $D$.
\end{theo}

\begin{proof}
If $f$ is a $C^2$-smooth function, it follows from Lemmas \ref{rellich}, \ref{lem:3.6} and \ref{L3} that
\be\nonumber
&&\int_{\Gamma} \mu|\partial_{\nu}u|^2\nu_3 + (\lambda + \mu) |\nabla \cdot u|^2 \nu_3 {\rm d}s \\ \nonumber
&=&2\omega^2 \Big\{ \int_{|\xi|<\kappa_{\rm p}} \beta^2(\xi) |A_{\rm p}(\xi)|^2\,{\rm d}\xi+
\int_{|\xi|<\kappa_{\rm s}} \gamma^2(\xi) |\textbf{A}_{\rm s}(\xi)|^2\,{\rm d}\xi
\Big\}\\ \label{eq:3}
&=&0.
\en
The geometric assumption of $\Gamma$ implies that
\ben
\nu_3(x)=\frac{1}{\sqrt{1+|\nabla_{x^{\prime}}f|^2}}>C_L>0
\quad\mbox{on}\,\Gamma,
\enn
where $C_L$ is a constant depending on $L$ only.
Hence, we get $u=\partial_\nu u=0$ on $\Gamma$. As a consequence of the unique continuation in elasticity, it holds that $u\equiv0$ in $D$. This proves the uniqueness for $C^2$-smooth functions. Finally,
the proof can be completed by applying Ne$\breve{\rm c}$as' approach in \cite[Chapter 5]{N-les} of approximating a Lipschitz graph by smooth surfaces. We refer to  \cite{eh-mmmas12} for the application of the Ne$\breve{\rm c}$as' approximation theory to bi-periodic surfaces and  \cite{eh-siaa12} for rough surfaces in two dimensions in elasticity.
\end{proof}

In the proof of Theorem \ref{uniqueness}, the relation (\ref{eq:3}) is derived based on the important identity (\ref{eq:4}). Combined with the identity \eqref{eq:a}, this identity  will be used to prove the existence of solutions to the rough surface scattering problems. We remark that, for the uniqueness proof only, the relation (\ref{eq:3}) can be also obtained in a more straightforward way without using (\ref{eq:4}), which is given as follows.

\begin{proof}
By using Lemma \ref{rellich} and Lemma \ref{lem:3.6} (i), we obtain for each fixed $b>f_+$ that
\be\label{rhs}
\int_{\Gamma} \mu|\partial_{\nu}u|^2\nu_3 + (\lambda + \mu) |\nabla \cdot u|^2 \nu_3 {\rm d}s
= \int_{\Gamma_b} (2\Re(Tu \cdot \partial_3 \bar{u}) - \nu_3 \mathcal{E}(u,\bar{u}) + \omega^2|u|^2) {\rm d}s.
\en
It suffices to show that the right hand side of \eqref{rhs} vanishes.
By Lemma \ref{L3},
\begin{align}
u=u^{\rm sc} &= \int_{|\xi|\geq \kappa_{\rm p}} A_{\rm p}(\xi)\,(\xi,\beta)^{\top} e^{{\rm i}\beta (x_3-b)}
e^{{\rm i} \xi \cdot x^{\prime}}{\rm d} \xi + \int_{|\xi|\geq \kappa_{\rm s}}\boldsymbol A_{{\rm s}}(\xi)\, e^{{\rm i}\gamma (x_3-b)} e^{{\rm i} \xi \cdot x^{\prime}}{\rm d} \xi, \quad x_3\geq b,\notag\\ \label{eq:5}
\widehat{\partial_3u}(\xi, c)&={\rm i} \beta(\xi)\,A_{\rm p}(\xi)\,(\xi,\beta)^{\top} e^{{\rm i}\beta (c-b)}+{\rm i}\gamma(\xi)\,\boldsymbol A_{{\rm s}}(\xi)\, e^{{\rm i}\gamma (c-b)},\quad c>b.
\end{align}
Since the right hand side of \eqref{rhs} does not depend on the choice of $b$, we have for each $c>b$ that
\begin{align}\label{eq:7}
\int_{\Gamma_b} 2\Re(Tu \cdot \partial_3 \bar{u}) - \nu_3 \mathcal{E}(u,\bar{u}) + \omega^2|u|^2 {\rm d}s=
 \int_{\Gamma_c} 2\Re(Tu \cdot \partial_3 \bar{u}) - \nu_3 \mathcal{E}(u,\bar{u}) + \omega^2|u|^2 {\rm d}s.
\end{align}

We first prove the vanishing of the first term on the right hand side of the above identity as $c\rightarrow +\infty$.
Using (\ref{DtN}), (\ref{eq:5}) and Lemma \ref{L3}, we obtain
\begin{align}\nonumber
&\Re\int_{\Gamma_c}Tu \cdot \partial_3 \bar{u}{\rm d}s = \Re\int_{\mathbb R^2}\widehat{\mathcal Tu} \cdot  \overline{\widehat{\partial_3 u}}{\rm d}\xi\\ \nonumber
& =\Im\int_{\mathbb R^2} M(\xi)\Big(A_{\rm p}(\xi)(\xi,\beta)^\top e^{{\rm i}\beta (c-b)} + \boldsymbol A_{\rm s}(\xi) e^{{\rm i}\gamma (c-b)}\Big)\cdot \overline{\Big(\beta A_{\rm p}(\xi)(\xi,\beta)^\top e^{{\rm i}\beta (c-b)} + \gamma \boldsymbol A_{\rm s}(\xi) e^{{\rm i}\gamma (c-b)}\Big)}{\rm d}\xi\\ \nonumber
& =\Im\int_{|\xi|\geq \kappa_{\rm p}} M(\xi)\Big(A_{\rm p}(\xi)(\xi,\beta)^\top\Big)\cdot\overline{\Big( \beta A_{\rm p}(\xi)(\xi,\beta)^\top\Big)}e^{-2(c-b)\sqrt{|\xi|^2 - \kappa_{\rm p}^2}}{\rm d}\xi\\ \nonumber
&\quad +\Im \int_{|\xi|\geq \kappa_{\rm p}} M(\xi)\Big(A_{\rm p}(\xi)(\xi,\beta)^\top\Big)\cdot\overline{\Big( \gamma \boldsymbol A_{\rm s}(\xi)\Big)} e^{-(c-b)\sqrt{|\xi|^2 - \kappa_{\rm p}^2}}e^{-(c-b)\bar{\gamma}}{\rm d}\xi\\ \nonumber
&\quad +\Im \int_{|\xi|\geq \kappa_{\rm p}}M(\xi) \boldsymbol A_{\rm s}(\xi)\cdot\overline{\Big( \beta A_{\rm p}(\xi,\beta)^\top\Big)}e^{-(c-b)\sqrt{|\xi|^2 - \kappa_{\rm p}^2}}e^{-(c-b)\bar{\gamma}}{\rm d}\xi\\ \label{eq:6}
&\quad+ \Im \int_{|\xi|\geq \kappa_{\rm s}} M(\xi)\boldsymbol A_{\rm s}(\xi)\cdot\overline{\Big( \gamma \boldsymbol A_{\rm s}(\xi)\Big)}e^{-2(c-b)\sqrt{|\xi|^2 - \kappa_{\rm s}^2}}{\rm d}\xi,
\end{align}
where the matrix $M$ is given by (\ref{M}), and the dot denotes the inner product over $\R^2$.
For each $\epsilon > 0$, there exists a sufficiently small $\delta>0$, which does not depend on $c$, such that
\begin{align*}
\Im\int_{\kappa_{\rm p} \leq |\xi|\leq \kappa_{\rm p}+\delta} M(\xi)\Big(A_{\rm p}(\xi)(\xi,\beta)^\top\Big) \cdot\overline{\Big(\beta A_{\rm p}(\xi,\beta)^\top\Big)}e^{-2(c-b)\sqrt{|\xi|^2 - \kappa_{\rm p}^2}}{\rm d}\xi < \epsilon.
\end{align*}
On the other hand, we have
\[
\lim_{c\rightarrow +\infty}\int_{|\xi|\geq \kappa_{\rm p}+\delta}M(\xi)\Big(A_{\rm p}(\xi)(\xi,\beta)^\top\Big) \cdot\overline{\Big(\beta A_{\rm p}(\xi,\beta)^\top\Big)}e^{-2(c-b)\sqrt{|\xi|^2 - \kappa_{\rm p}^2}}{\rm d}\xi = 0,
\]
since it is an exponentially decaying function as $c \rightarrow +\infty$. Hence, the first term on the right hand side of (\ref{eq:6}) tends to zero as $c\rightarrow\infty$.
The vanishing of the remaining terms on the right hand side of (\ref{eq:6}) and those of (\ref{eq:7}) can be proved similarly.
This proves the vanishing of (\ref{rhs}), due to the relation (\ref{eq:7}) and the arbitrariness of $c>b$.
\end{proof}

\section{Existence}\label{existence}
\setcounter{equation}{0}

In this section, we discuss the existence of solutions to the scattering problems where the flat surfaces are locally perturbed.

\subsection{Scattering from flat surfaces}

The propagation and reflection of elastic waves in a homogeneous half-space have been of significant interest in the classical seismology. The analytical solutions of this problem are frequently used in the literature for various purposes. In this section, we assume that $\Gamma = \Gamma_0$ (i.e., $b=0$) is a rigid flat surface. In this case, the total field $u = u^{\rm in} + u^{\rm re}$ consists of two parts: the incident field $u^{\rm in}$ and the reflected field $u^{\rm re}$ which solves the boundary value problem
\[
\mu\Delta{u}^{\rm re}+(\lambda+\mu)\nabla
\nabla\cdot{u}^{\rm re} + \omega^2{u}^{\rm re}=0\quad\text{in}\quad U_0,\quad
{u}^{\rm re} = -{u}^{\rm in}\quad\text{on}\quad \Gamma_0.
\]

If $u^{\rm in}$ is a compressional plane wave of the form \eqref{icw}, we have
\begin{align}\nonumber
u^{\rm re}=u^{\rm re}_{\rm p} &= -
\frac{ (\alpha,\gamma)^\top\cdot d}{\beta\gamma+|\alpha|^2}(\alpha, \beta)^\top e^{{\rm i} (\alpha\cdot x^{\prime}+ \beta x_3)} \notag\\ \label{re-p}
&\quad - \frac{1}{\beta\gamma+|\alpha|^2} \Big[(\alpha, \gamma)^\top\times \Big(d\times(\alpha, \beta)^\top\Big)\Big] e^{{\rm i} (\alpha\cdot x^{\prime}+\gamma x_3)},
\end{align}
where
\begin{align*}
\alpha = \kappa_{\rm p}(\sin\theta\cos\varphi, \sin\theta\sin\varphi),\quad
\beta = \sqrt{\kappa_{\rm p}^2 - |\alpha|^2}, \quad \gamma = \sqrt{\kappa_{\rm s}^2 - |\alpha|^2}.
\end{align*}
For the shear incident plane wave \eqref{isp} with $d\cdot d^\bot_j=0$ ($j=1,2$), it holds that
\begin{align}\nonumber
u^{\rm re}=u^{\rm re}_{{\rm s},j} &=
-\frac{ (\alpha,\gamma)^\top\cdot d_j^\bot}{\beta\gamma+|\alpha|^2}(\alpha, \beta)^\top e^{{\rm i} (\alpha\cdot x^{\prime}+ \beta x_3)} \\ \label{re-s}
&\quad - \frac{1}{\beta\gamma+|\alpha|^2} \Big[(\alpha, \gamma)^\top\times \Big(d_j^\bot\times(\alpha, \beta)^\top\Big)\Big] e^{{\rm i} (\alpha\cdot x^{\prime}+\gamma x_3)},
\end{align}
where
\begin{align*}
\alpha = \kappa_{\rm s}(\sin\theta\cos\varphi, \sin\theta\sin\varphi),\quad
\beta = \sqrt{\kappa_{\rm p}^2 - |\alpha|^2}, \quad \gamma = \sqrt{\kappa_{\rm s}^2 - |\alpha|^2}.
\end{align*}
Thus, if $u^{\rm in}$ takes the general form \eqref{inc}, by linear superposition, the reflected wave is given by
\begin{align}\label{reflect}
{u}^{\rm re}(x) = c_{\rm p}{u}^{\rm re}_{\rm p}(x) + c_{{\rm s}, 1}{u}^{\rm re}_{{\rm s}, 1}(x)
 +  c_{{\rm s}, 2}{u}^{\rm re}_{{\rm s}, 2}(x).
\end{align}

The expressions of \eqref{re-p} and \eqref{re-s} follow directly from the UPRC \eqref{uprc} with $\hat{u}^{\rm re}(\xi,0)=-\hat{u}^{\rm in}(\xi,0)$. They can be also obtained from the upward Rayleigh expansion \eqref{Ray} with $u_n^{\rm sc}(b)=-u_n^{\rm in}(b)$ for $n=(0,0)$ and $u_n^{\rm sc}(b)=0$ for $|n|\neq 0$.
These analytical solutions in a half-space indicate that, in general case, a compressional (resp. shear) plane wave reflects back to the domain as a sum of both compressional and shear waves.

Below we derive the reflected wave corresponding to the point source incidence
\eqref{gtn} with the source position $y\in \R^3_+$. In this case, the total field $u =u^{\rm in}+u^{\rm re}$ coincides with the Green's tensor $\text{G}_{\rm H}( x,  y)$ to the first boundary boundary value problem of the Navier in a half space, i.e.,, $\text{G}_{\rm H}( x,  y)$ satisfies
\begin{align*}
 \mu\Delta_y\text{G}_{\rm H}( x,  y)+(\lambda+\mu)\nabla_y
\nabla_y\cdot\text{G}_{\rm H}( x,  y) + \omega^2\text{G}_{\rm H}( x,  y)&=-\delta( x -  y) \text{I} \quad &&\text{in} \quad U_0,\quad  x\neq  y,\\
\text{G}_{\rm H}( x,  y) &= 0  \quad &&\text{on} \quad \Gamma_0.
\end{align*}
Before stating the expression of $\text{G}_{\rm H}( x,  y)$, we introduce the outgoing Kupradze radiation condition for the scattered field $u^{\rm sc}$ in a half space.
\begin{defi}\label{def:radiation}
An upward radiating solution to the Navier equation \eqref{tf} with $D=U_0$ is said to satisfy the half-space Kupradze radiation condition if its compressional part $\varphi$ and shear part $\psi$ satisfy the Sommerfeld
radiation condition as follows:
\begin{align}\label{rc}\begin{split}
\varphi(x)=O(r^{-1}),\quad \partial_r\varphi -{\rm
i}\kappa_{\rm p}\varphi=o(r^{-1}),\\
\psi(x)=O(r^{-1}),\quad \partial_r\psi -{\rm
i}\kappa_{\rm s}\psi=o(r^{-1}),
\end{split}\end{align}
uniformly in all $x \in \{| x|>R\}\cap U_0$ as $r:=|x|\rightarrow\infty$.
\end{defi}
In the following lemma, $G$ is the free-space Green tensor given by $(\ref{gtn})$ and $\tilde{x} = (x^{\prime}, -x_3)$ for $x=(x', x_3)\in \R^3$.

\begin{lemm}\label{green}
(i) The half-space Green tensor $\text{G}_{\rm H}(\cdot,  y)$ ($y_3>0$) can be expressed as
\begin{align}\label{GH}
\text{G}_{\rm H}( x,  y) &= \text{G} ( x,  y) - \text{G}( \tilde{x},  y) + \text{U} ( x, y),\quad x_3>0,\quad x\neq y.
\end{align}
where $U(x,y)$ is given by
\begin{align*}
 \text{U}(x, y)&=\frac{\rm i}{2\pi\omega^2}\int_{\mathbb R^2}
\frac{1}{\beta\,\gamma+| \xi|^2}  \Big(\widetilde{M}_{\rm p}(\xi) e^{{\rm i} \xi\cdot (y^{\prime}-x^{\prime})}e^{{\rm i}\beta y_3} (e^{{\rm i}\beta x_3} - e^{{\rm i}\gamma x_3})\notag\\
&\qquad \qquad\qquad\qquad\qquad+ \widetilde{M}_{\rm s}(\xi) e^{{\rm i} \xi\cdot (y^{\prime}-x^{\prime})}e^{{\rm i}\gamma y_3} (e^{{\rm i}\beta x_3} - e^{{\rm i}\gamma x_3})
  \Big) {\rm d} \xi
\end{align*}
with
\begin{align*}
\widetilde{M}_{\rm p}(\xi)= \begin{bmatrix}
\gamma \xi_1^2       &  \gamma \xi_1\xi_2   & \xi_1 |\xi|^2\\
\gamma \xi_1\xi_2    & \gamma\xi_2^2        & \xi_2 |\xi|^2\\
\beta\gamma \xi_1 & \beta\gamma\xi_2  & \beta |\xi|^2
\end{bmatrix},\quad
\widetilde{M}_{\rm s}(\xi)=
\begin{bmatrix}
-\gamma\xi_1^2 & -\gamma \xi_2^2               & \beta\gamma\xi_1\\
-\gamma\xi_1 \xi_2             & -\gamma\xi_2^2  & \beta\gamma \xi_2\\
\xi_1|\xi|^2           & \xi_2|\xi|^2            & -\beta|\xi|^2
\end{bmatrix}.
\end{align*}
(ii) The columns of the matrix function $\text{G}_{\rm H}(x, \cdot)$ and the rows of the matrix function $\text{G}_{\rm H}( \cdot, y)$ satisfy the half-space Kupradze radiation condition.
\end{lemm}

We remark that the first two terms on the right hand side of \eqref{GH}, i.e., $\text{G}(x,y)-\text{G}(\tilde{x},y)$ does not satisfy the Navier equation in $x_3>0$, although it vanishes on $x_3=0$. We refer to \cite{a-siaa01} for the expression of $U$ in two dimensions.

\begin{proof}
Since $\text{G}_{\rm H}(\cdot,  \cdot)$ is symmetric, we fix $x_3>0$ and take $y$ as the variable in our proof.

(i) Taking the Fourier transform of $g_{\rm p}(x, y)$ and $g_{\rm s}( x, y)$ (see \eqref{gn}) with respect to the variable $y^{\prime}\in \R^2$ gives
\[
\hat{g}_{\rm p}( x, (\xi, y_3)) = \frac{\rm i}{2\beta} e^{{\rm i} \beta |x_3-y_3|} e^{-{\rm i}\xi_1 x_1} e^{-{\rm i}\xi_2 x_2},\quad
\hat{g}_{\rm s}(x, (\xi, y_3)) = \frac{\rm i}{2\gamma} e^{{\rm i} \gamma |x_3-y_3|} e^{-{\rm i}\xi_1 x_1} e^{-{\rm i}\xi_2 x_2}.
\]
The Dirichlet boundary condition on $y_3 = 0$ gives the relation
\be\nonumber
\text{U}( x, y) &=&-\text{G}(x,y)+\text{G}(\tilde{x},y)\\ \nonumber
 &=&- \frac{1}{\omega^2}\nabla_{ y}\nabla^\top_{ y}(g_{\rm s}( x, y)-g_{\rm p}( x,
y)) + \frac{1}{\omega^2}\nabla_{ y}\nabla^\top_{ y}(g_{\rm s}( \tilde{x}, y)-g_{\rm p}( \tilde{x},
y))\\ \label{U}
&=& \frac{1}{\omega^2}\nabla_{y}\nabla^\top_{ y}(g_{\rm s}( \tilde{x}, y)-g_{\rm s}( x,
y)) - \frac{1}{\omega^2}\nabla_{ y}\nabla^\top_{ y}(g_{\rm p}( \tilde{x}, y)-g_{\rm p}(x,
y)).
\en
Therefore, the Fourier transform of $\text{U}( x, y)$ on $y_3=0$, which we denote by $\hat{\text{U}}(x, \xi):=(\hat{\text{U}}( x, (\xi, 0))_{ij}$, takes the form
\begin{align*}
\hat{\text{U}}(x, \xi)=\frac{\rm i}{\omega^2}e^{-{\rm i}\xi_1 x_1} e^{-{\rm i}\xi_2 x_2} \Big(e^{{\rm i} \beta x_3} - e^{{\rm i} \gamma x_3}\Big) V(\xi),\quad V(\xi):=
\begin{bmatrix}
0 & 0 & \xi_1 \\
0 & 0 & \xi_2 \\
\xi_1 & \xi_2 & 0
\end{bmatrix}.
\end{align*}
Consequently, we have from the UASR \eqref{uprc} that
\begin{align*}
 \text{U}&=\frac{\rm i}{2\pi\omega^2}\int_{\mathbb R^2}\Big\{
\frac{1}{\beta\,\gamma+| \xi|^2}  \Big(M_{\rm p}(\xi) e^{{\rm i} \xi\cdot (y^{\prime}-x^{\prime})}e^{{\rm i}\beta y_3} (e^{{\rm i}\beta x_3} - e^{{\rm i}\gamma x_3})\notag\\
&\quad + M_{\rm s}(\xi) e^{{\rm i} \xi\cdot (y^{\prime}-x^{\prime})}e^{{\rm i}\gamma(y_3-b)} (e^{{\rm i}\beta x_3} - e^{{\rm i}\gamma x_3})
  \Big\}\, V(\xi) {\rm d} \xi\notag\\
 &=\frac{\rm i}{2\pi\omega^2}\int_{\mathbb R^2}\Big\{
\frac{1}{\beta\,\gamma+| \xi|^2}  \Big(\widetilde{M}_{\rm p}(\xi) e^{{\rm i} \xi\cdot (y^{\prime}-x^{\prime})}e^{{\rm i}\beta y_3} (e^{{\rm i}\beta x_3} - e^{{\rm i}\gamma x_3})\notag\\
&\quad + \widetilde{M}_{\rm s}(\xi) e^{{\rm i} \xi\cdot (y^{\prime}-x^{\prime})}e^{{\rm i}\gamma y_3} (e^{{\rm i}\beta x_3} - e^{{\rm i}\gamma x_3})
  \Big\} {\rm d} \xi, \quad \quad y_3>0,
\end{align*}
where $M_{\rm p}$ and $M_{\rm s}$ are given respectively in \eqref{Mp} and \eqref{Ms}, and
\begin{align*}
\widetilde{M}_{\rm p}(\xi)&=M_{\rm p}(\xi)\,V(\xi) =\begin{bmatrix}
\gamma \xi_1^2       &  \gamma \xi_1\xi_2   & \xi_1 |\xi|^2\\
\gamma \xi_1\xi_2    & \gamma\xi_2^2        & \xi_2 |\xi|^2\\
\beta\gamma \xi_1 & \beta\gamma\xi_2  & \beta |\xi|^2
\end{bmatrix},\\
\widetilde{M}_{\rm s}(\xi)&=M_{\rm s}(\xi)\,V(\xi)=
\begin{bmatrix}
-\gamma\xi_1^2 & -\gamma \xi_2^2               & \beta\gamma\xi_1\\
-\gamma\xi_1 \xi_2             & -\gamma\xi_2^2  & \beta\gamma \xi_2\\
\xi_1|\xi|^2           & \xi_2|\xi|^2            & -\beta|\xi|^2
\end{bmatrix}.
\end{align*}

(ii) To prove the half-space Kupradze radiation condition of $\text{G}_{\rm H}$, we adopt the two-dimensional arguments of Arens \cite[Theorem 4.5]{Arens00}. Let
\ben
\text{U}_{\rm p}(x,y)=
\frac{\rm i}{2\pi\omega^2}\int_{\mathbb R^2}
\frac{1}{\beta\,\gamma+| \xi|^2}  \Big(\widetilde{M}_{\rm p}(\xi) e^{{\rm i} \xi\cdot (y^{\prime}-x^{\prime})}e^{{\rm i}\beta y_3} (e^{{\rm i}\beta x_3} - e^{{\rm i}\gamma x_3})
  \Big) {\rm d} \xi,\\
\text{U}_{\rm s}(x,y)=
\frac{\rm i}{2\pi\omega^2}\int_{\mathbb R^2}
\frac{1}{\beta\,\gamma+| \xi|^2}  \Big(\widetilde{M}_{\rm s}(\xi) e^{{\rm i} \xi\cdot (y^{\prime}-x^{\prime})}e^{{\rm i}\beta y_3} (e^{{\rm i}\gamma x_3} - e^{{\rm i}\gamma x_3})
  \Big) {\rm d} \xi.
\enn
It suffices to verify that $\text{U}_\alpha$ ($\alpha=\rm p, \rm s$) fulfills the Sommerfeld radiation condition specified in Definition \ref{def:radiation}. Note that $(\Delta_y+k_\alpha^2) \text{U}_\alpha(x,y)=0$ for $\alpha=\rm p,\rm s$ and all $y\in\R^3_+\backslash\{x\}$.
Since $\text{U}=\text{U}_{\rm p}+\text{U}_{\rm s}$, it follows from \eqref{U} that
\ben
\text{U}_{\rm p}(x,y)=\frac{1}{\omega^2}\nabla_{y}\nabla^\top_{ y}(g_{\rm s}( \tilde{x}, y)-g_{\rm s}( x,
y)) - \frac{1}{\omega^2}\nabla_{ y}\nabla^\top_{ y}(g_{\rm p}( \tilde{x}, y)-g_{\rm p}(x,
y))-\text{U}_{\rm s}(x,y),\quad y_3=0.
\enn
Direct calculations show that $|g_{\rm \alpha}( \tilde{x}, y)-g_{\rm \alpha}( x,
y)|\leq C (1+x_3)(1+y_3) |x-y|^{-2}$ for all $x\neq y$ with $x,y\neq 0$ and $x_3,y_3\geq 0$ and all $\alpha={\rm p,\rm s}$.
 Hence, it follows from the interior estimate that
\be\label{decay}
w(x,y'):=\text{U}_{\rm p}(x,y)|_{y_3=0}\leq C\,(1+|y'|)^{-2}\quad\mbox{for some fixed}\quad x\in \R_+^3.
\en

Reviewing the UPRC and ASR for the Helmholtz equation, we obtain for $y_3>0$ that
\ben
\text{U}_{\rm p}(x,y)=2\int_{\Gamma_0} \frac{\partial g_{\rm p}(y,z)}{\partial z_3}w(x,z'){\rm d}s(z')
=\frac{1}{2\pi}\int_{\R^2} e^{{\rm i}\beta(\xi)\,y_3+{\rm i}\xi\cdot y' } \hat{w}(x,\xi)\,{\rm d}\xi.
\enn
We can then use the argument presented in \cite[Section 5]{CRZ} and \cite[Lemma 2.2 and Corollary 4.1]{HR2018} to conclude that the decay rate of \eqref{decay}  ensures the Sommerfeld radiating behavior of $\text{U}_{\rm p}$ as $|y|\rightarrow\infty$ in $y_3>0$. The Sommerfeld radiation condition of $\text{U}_{\rm s}$ can be proceeded analogously. We note that the arguments of \cite{CRZ,HR2018} present the decaying behavior of the scattered field for the two-dimensional rough surface scattering problems due to a compact source term or a point source incidence and can be readily carried over to the three-dimensional case.
\end{proof}

\subsection{Scattering from locally perturbed flat surfaces}

In this section we consider the existence of weak solutions for the scattering problem \eqref{tf} and \eqref{uprc}, where $\Gamma$ is a locally perturbed flat surface. Without loss of generality, we assume that $\Gamma$ coincides with the ground plane $\Gamma_0:=\{x_3=0\}$ in $|x|>R$ for some $R>\max_{x\in \Gamma}\{x_3\}$. Hence, the domain $D$ above $\Gamma$ is a locally perturbed half space. In this case, as can be seen from the subsequent subsections, we can propose an equivalent variational formulation in a bounded domain by truncating the unbounded domain $D$ with a transparent boundary condition and then applying the Fredholm alternative. The reduction to a bounded domain has significantly simplified the arguments for globally perturbed scattering problems, because the compact embedding of $H^1$ into $L^2$ is in general not valid in an unbounded domain.

Specifically, we consider to cases:
\begin{itemize}
\item[(i)] The perturbation lies entirely below the ground plane, i.e., $\Gamma\cap\{x_3>0\}=\emptyset$.
\item[(ii)] The perturbation is allowed to occur in the upper half space, i.e., $\Gamma\cap\{x_3>0\}\neq \emptyset$.
\end{itemize}
Note that in the literature, Case (i) is referred to as an open cavity scattering problem in acoustics and electromagnetism, whereas Case (ii) is known as an overfilled cavity scattering problem. The above two cases will be investigated in the following two subsections separately. In particular, the existence result of Theorem \ref{exist} has improved the well-posedness of acoustic cavity scattering problems \cite{lw-jcp13}, while Theorem \ref{th5.12} has generalized the two-dimensional result \cite{eh-siaa12} to three dimensions. Some open questions will be discussed in Remark \ref{Rem:5.4}.

\subsubsection{Case (i): perturbation beneath the ground plane}

For simplicity, we assume that $\Omega=\Omega\cap\{x_3<0\}$ is connected. The problem geometry is shown in Figure \ref{fig3}. If $\Omega$ is disconnected, one can apply our variational argument to each connected set of $\Omega$. Let $\Lambda_0$ be the aperture of $\Omega$ and $S$ be the boundary of $\Omega$ in the lower half space. We have $\partial \Omega=\Lambda_0\cap S$ and $D = \Omega \cup U_0 \cup \Lambda_0$. Let $\Gamma_0^c=\Gamma_0\backslash \Lambda_0$ and $\Gamma = S\cup \Gamma_0^c$. We assume that the scattering surface $\Gamma$ (especially the boundary $S$) is a Lipschitz continuous surface but not necessary the graph of some function.

\begin{figure}
\centering
\includegraphics[width=15cm,viewport=0.pt 180.pt 800.pt
  440.pt,clip]{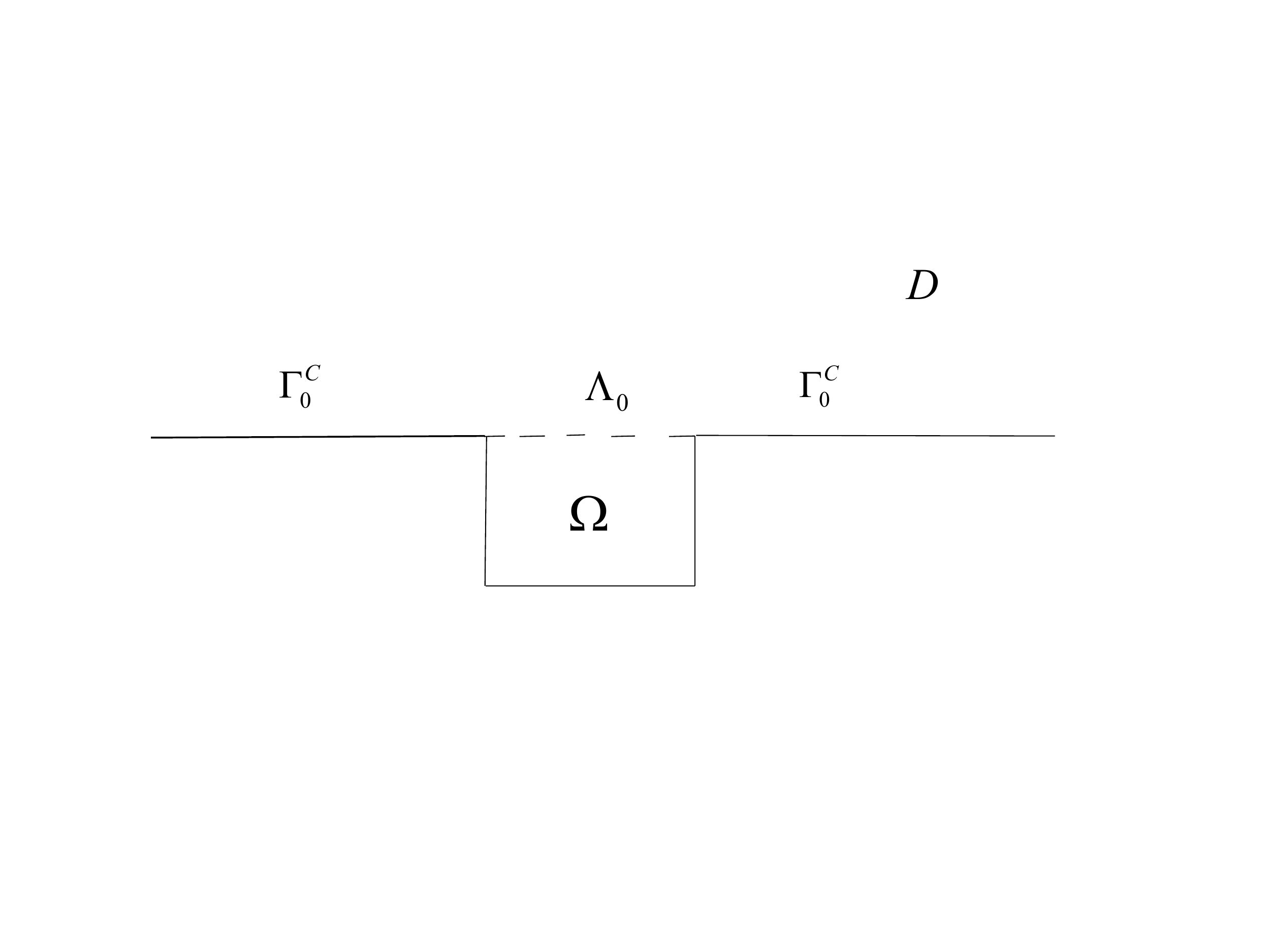}
\caption{The problem geometry of a local perturbation of the ground plane which lies entirely in the lower half space.}
\label{fig3}
\end{figure}

Introduce the functional space
\[
\tilde{H}^{1/2}(\Lambda_0)^3 = \{v: \mbox{the zero-extension of $v$ from $\Lambda_0$ to $\Gamma_0$ belongs to $H^{1/2}(\mathbb R^2)^3$}\}.
\]
Denote by $H^{-1/2}(\Lambda_0)^3$ the dual space of $\tilde{H}^{1/2}(\Lambda_0)^3$. We propose a variational formulation over the Hilbert space
\[
H_S^1(\Omega)^3 = \{ u \in H^1(\Omega)^3:  u = 0 ~\text{on} ~S,\, u|_{\Lambda_0} \in \tilde{H}^{1/2}(\Lambda_0)^3\}.
\]

Consider a downward propagating pressure wave of the form
\begin{align}\label{2.15}
 u_{\rm pg}^{\rm in}(x) = \int_{\mathbb R^2}\frac{1}{\beta\,\gamma+| \xi|^2} M^{(D)}_{\rm p}(\xi)( \xi, -\beta)^{\top}g( \xi) e^{{\rm i}( \xi \cdot x^{\prime}-\beta (x_3-b))} {\rm d} \xi,\quad x\in S_b,
\end{align}
where $b>0$ and $g$ belongs to space of distributions $\mathcal{D}^{\prime}(\mathbb R^2)$ such that $\text{supp}(g)\subset\{|\xi|<\kappa_{\rm p}\}$. Alternatively, we may consider an incident shear wave of the form
\begin{align}\label{2.16}
 u_{\rm sg}^{\rm in}(x) = \int_{\mathbb R^2}\frac{1}{\beta\,\gamma+| \xi|^2} M^{(D)}_{\rm s}(\xi)(( \xi, -\gamma) \times  {\boldsymbol q}( \xi))^{\top} e^{{\rm i}( \xi \cdot x^{\prime}-\gamma (x_3-b))} {\rm d} \xi,\quad x\in S_b,
\end{align}
where $\boldsymbol q\in \mathcal{D}^{\prime}(\mathbb R^2)^3$ is a vector distribution such that $\text{supp}(\boldsymbol q)\subset\{| \xi|<\kappa_{\rm s}\}$.
Here, the matrices $M^{(D)}_{\rm p}$ and $M^{(D)}_{\rm s}$ are defined in \eqref{MpsD}. By direct calculations it is easy to verify that both $u_{\rm pg}^{\rm in}(x)$ and $u_{\rm sg}^{\rm in}(x)$ satisfy the Navier equation \eqref{ifne}.

\begin{rema} We remark that the set of incident compressional (resp. shear) waves \eqref{2.15} (resp. \eqref{2.16}) includes the compressional (resp. shear) plane wave \eqref{icw} (resp. \eqref{isp}).
In fact, since the plane waves
 can be rewritten as
 \ben
 u_{\rm p}^{\rm in} = \frac{1}{{\rm i}\kappa_{\rm p}}\nabla e^{{\rm i}\kappa_{\rm p}x\cdot d},\quad
 {u}^{\rm in}_{{\rm s}, j}(x) = \frac{1}{{\rm i}\kappa_{\rm s}}d \times q_j e^{{\rm i}\kappa_{\rm s}x\cdot d} = q_j \nabla\times e^{{\rm i}\kappa_{\rm s}x\cdot d},\quad j=1,2,
 \enn
 where $q_j$ $(j=1,2)$ are unit vectors in $\mathbb R^3$ satisfying $q_1\cdot q_2=0$ and $q_j\cdot d=0$, it follows from the downward ASR \eqref{dprc} that $u_{\rm p}^{\rm in}$ and $u_{\rm s}^{\rm in}$ can be also formulated respectively as the representations \eqref{2.15} and \eqref{2.16} with \ben
 g(\xi) = \frac{1}{2\pi\kappa_{\rm p}}\widehat{e^{{\rm i}\kappa_{\rm p}x\cdot d}}(\xi)|_{\Gamma_b} = \frac{e^{{\rm i}\kappa_{\rm p} x_3b}}{\kappa_{\rm p}}\delta(\xi - \kappa_{\rm p}d^{\prime}),\quad
 \textbf{q}=\textbf{q}_j(\xi) = \frac{e^{{\rm i}\kappa_{\rm s} x_3b}}{\kappa_{\rm s}} q_j \delta(\xi - \kappa_{\rm s}d^{\prime}).
 \enn
\end{rema}

Let $u^{\rm in}$ be an incoming wave of the form
\begin{align}\label{incc}
u^{\rm in}(x) = c_{\rm p} u_{\rm pg}^{\rm in}(x) + c_{\rm s}u_{\rm sg}^{\rm in}(x),\quad c_{\rm p}, c_{\rm s}\in \C.
\end{align}
Multiplying the complex conjugate of a test function $ \phi \in H^1_S(\Omega)^3$ on both sides of the Navier equation
\ben
\mu\Delta{u}+(\lambda+\mu)\nabla
\nabla\cdot{u} + \omega^2{u}=0\quad &\text{in}~ \Omega,
\enn
integrating over $\Omega$ and using the integration by part together with the DtN map \eqref{traction}, we deduce an equivalent variational problem: find $u \in H^1_S(\Omega)^3$ such that
\begin{align}\label{vpc}
 B( u,  \phi)= \int_{\Lambda_0}  p \cdot \bar{ \phi}\, {\rm d}x'  \quad \forall ~ \phi \in H^1_S(\Omega)^3,
\end{align}
where $p:= Tu^{\rm in} - \mathcal{T}u^{\rm in}\in H^{-1/2}(\Lambda_0)^3$ and
\begin{align*}
B( u,  \phi):= \int_{\Omega} \mathcal{E}( u, \bar{ \phi}) - \omega^2  u \cdot  \bar{ \phi}\, {\rm d}  x - \int_{\Lambda_0}  \bar{ \phi} \cdot \mathcal{T} \tilde{{u}}\, {\rm d}x'.
\end{align*}
Note that the symbol $\tilde{f}$ stands for the zero extension of $f$ from $\Lambda_0$ to $\Gamma_0$. In deriving \eqref{vpc}, we have used the following identity on $\Lambda_0$:
\ben
Tu&=&Tu^{\rm sc}+Tu^{\rm in}=T\tilde{u}^{\rm sc}+Tu^{\rm in}\\
&=&\mathcal{T}\tilde{u}^{\rm sc}+Tu^{\rm in}=
\mathcal{T}\tilde{u}-\mathcal{T}\tilde{u}^{\rm in}+Tu^{\rm in}\\
&=&\mathcal{T}\tilde{u}-p.
\enn

Moreover, using \eqref{DtN}, one can derive an explicit form of $p$:
\ben
p(x')=\left\{\begin{array}{lll}
\displaystyle\int_{\mathbb R^2} \frac{\rm i}{\kappa_{\rm p}} \frac{2\omega^2\beta}{| \xi|^2+\beta\gamma}(- \xi, \gamma)^{\top}e^{{\rm i} \xi \cdot x^{\prime}+{\rm i}\beta b}g( \xi){\rm d} \xi, &&\quad\mbox{if}\quad u^{\rm in}=u^{\rm in}_{\rm pg},\\
\displaystyle\int_{\mathbb R^2} \frac{\rm i}{\kappa_{\rm s}} \frac{2\omega^2\gamma}{| \xi|^2+\beta\gamma} {\boldsymbol q}( \xi)^{\top}\times( \xi, -\beta)^{\top}e^{{\rm i} \xi \cdot x^{\prime}+{\rm i}\gamma b}{\rm d} \xi, &&\quad\mbox{if}\quad u^{\rm in}=u^{\rm in}_{\rm sg}.
\end{array}\right.
\enn
By the trace theorem $\| {u}\|_{\tilde{H}^{1/2}(\Lambda_0)^3} \leq C\, \|
{u}\|_{H^1(\Omega)^3}$ for all ${u}\in H^1_{S}(\Omega)^{3}$ and the boundedness of the DtN map $\mathcal{T}$ (see the second assertion in Lemma \ref{dtn}),
 there exists a continuous linear operator $\mathcal B: H_S^{1}(\Omega)^3 \rightarrow H_S^{-1}(\Omega)^3:= (H_S^{1}(\Omega)^3)'$ associated with the sesquilinear form $B$ such that
\begin{align*}
B( u,  \phi) = (\mathcal B  u,  \phi), \quad \forall \phi \in H_S^{1}(\Omega)^3.
\end{align*}
Hence, the variational formulation \eqref{vpc} can be rewritten as
\begin{align}\label{4.9}
\mathcal B  u = \mathcal F,
\end{align}
where $\mathcal F \in H_S^{-1}(\Omega)^3$ is defined by the right-hand side of \eqref{vpc}.

\begin{theo}\label{exist}
For incoming waves of the form \eqref{incc}, there always exists a solution $u \in H_S^{1}(\Omega)^3$ to the variational problem \eqref{vpc}. Moreover, this solution can be extended from $\Omega$ to $D$ as a solution of our scattering problem  \eqref{tf} and \eqref{uprc} in $H^1_{loc}(D)$, which can be split as $u=u^{\rm in}+u^{\rm re}+v^{\rm sc}$ in $D$. Here $u^{\rm re}$ is the reflected wave caused by the rigid ground plane $x_3=0$ and $v^{\rm sc}$ satisfies the half-space Kupradze radiation condition (see Definition \ref{def:radiation}).
\end{theo}

\begin{proof}

We divide the proof into two steps: the first step is to prove the existence of the variational equation (\ref{4.9}) and the second step is to extend the solution of (\ref{4.9}) from $\Omega$ to $D$.

Step 1. By the Plancherel identity we have
\begin{align*}
&\Re\int_{\Lambda_0} \mathcal{T} \tilde{{u}} \cdot \bar{{u}} {\rm d}x^{\prime} = \Re\int_{\mathbb R^2} \mathcal{T} \tilde{{u}} \cdot \bar{\tilde{{u}}} {\rm d}x^{\prime} = \Re\int_{\mathbb R^2} \widehat{\mathcal{T} \tilde{{u}}} \cdot \bar{\hat{\tilde{{u}}}} {\rm d} \xi\\
&= \int_{| \xi|>K} {\rm i}M( \xi) \hat{\tilde{{u}}} \cdot \bar{\hat{\tilde{{u}}}} {\rm d} \xi + \int_{| \xi|\leq K} {\rm i}M( \xi) \hat{\tilde{{u}}} \cdot \bar{\hat{\tilde{{u}}}} {\rm d}\xi,
\end{align*}
where the matrix $M( \xi)$ defined in \eqref{M} and $K>0$ is sufficiently large such that $M(\xi)$ is positively definite for all $|\xi|>K$ (see Lemma \ref{dtn}).
Hence, the above identity implies that
\begin{align*}
-\Re\int_{\Lambda_0}  \mathcal{T} \tilde{{u}}\cdot \bar{u}{\rm d}x^{\prime} \geq -C \int_{| \xi|\leq K} |\hat{\tilde{u}}(\xi,0)|^2 {\rm d} \xi
\geq -C \int_{\mathbb R^2} |\hat{\tilde{u}}( \xi,0)|^2 {\rm d} \xi
 =  - C\,\| u\|^2_{L^2(\Lambda_0)^3}.
\end{align*}
Using the inequalities
\[
\| u\|^2_{L^2(\Lambda_0)^3} \leq \epsilon\, \| u\|^2_{H^1(\Omega)^3} + C_0(\epsilon)\| u\|^2_{L^2(\Omega)^3},\quad \epsilon>0,
\]
and
\[
\int_{\Omega}\mathcal{E}( u, \bar{ u}){\rm d} x + \int_{\Omega}| u|^2 {\rm d} x \geq C_1(\Omega) \| u\|^2_{H^1(\Omega)^3},
\]
we obtain
\begin{align*}
\Re B(u, u)\geq C_2\| u\|^2_{H^1(\Omega)^3} - C_3\| u\|^2_{L^2(\Omega)^3}.
\end{align*}
%From Lemma \ref{lem:1/2} we obtain
%\begin{align*}
%|\langle  p, \phi\rangle_{\Lambda_0}|\leq \| p\|_{H^{-1/2}(\Lambda_0)^3}
%\| \phi\|_{\tilde{H}^{1/2}(\Lambda_0)^3} \leq \| p\|_{H^{-1/2}(\Lambda_0)^3}
%\| \phi\|_{H^1_S(\Omega)^3}.
%\end{align*}
 Since the injection of $H_S^1(\Omega)^3$ into $L^2(\Omega)^3$ is compact, the above inequality shows that the sesquilinear form $B$ is strongly elliptic and thus the operator $\mathcal{B}$ is Fredholm with index zero. Hence, the operator equation \eqref{4.9} is solvable if its right-hand side $\mathcal F$ is orthogonal to all solutions $v\in H_S^1(\Omega)^3$ of the homogeneous adjoint equation $\mathcal B^* v =0$. Note that such $v$ satisfies
 \begin{align}\label{4.11}
(\mathcal B^* v,  \phi)_{L^2(\Omega)^3} = ( v, \mathcal B \phi)_{L^2(\Omega)^3} = \overline{B( \phi,  v)} = 0, \quad \forall \phi \in H^1_S(\Omega)^3.
\end{align}
 Furthermore, we can extend $v\in H_S^1(\Omega)^3$ to a solution of the Navier equation \eqref{tf} in the unbounded domain $U_0$ by setting
\begin{align*}
v(x) = \int_{\mathbb R^2} A_{\rm p}( \xi)( \xi, -\bar{\beta}( \xi))^{\top}e^{{\rm i}( \xi \cdot x^{\prime} - \bar{\beta}x_3)} + \boldsymbol A_{\rm s}( \xi) e^{{\rm i}( \xi \cdot x^{\prime} - \bar{\gamma}z)} {\rm d}\xi, \quad x_3>0,
\end{align*}
where $\boldsymbol A_{\rm s}(\xi)\in \mathbb C^{3\times3}$ satisfies the orthogonality relation $\boldsymbol A_{\rm s}( \xi) \cdot ( \xi, - \bar{\gamma}) = 0$ and
\begin{align*}
\hat{v}( \xi, 0)
 =
\begin{bmatrix}
\xi_1 & 1 & 0 & 0\\
\xi_2 & 0 & 1 & 0\\
- \bar{\beta} & 0 & 0 & 1
\end{bmatrix}
\begin{bmatrix}
A_{\rm p}(\xi)\\[5pt]
\boldsymbol A_{\rm s}^{\top}(\xi)
\end{bmatrix},\quad \xi\in \R^2.
\end{align*}
Analogously to Lemma \ref{L3}, it can be derived from (\ref{4.11})
that
\begin{align*}
A_{\rm p}( \xi) = 0 \quad \text{for} ~ |\xi|<\kappa_{\rm p}, \quad \boldsymbol A_{\rm s}( \xi) = 0 \quad \text{for} ~ | \xi|<\kappa_{\rm s}.
\end{align*}
Hence, if the incident wave has the form \eqref{2.15} with $\mbox{supp}(g)\subset\{\xi: |\xi|<\kappa_{\rm p}\}$, then
\begin{align*}
&\mathcal F(v)  = \int_{\mathbb R^2}\hat{p} \bar{\hat{ \tilde{v}}}{\rm d} \xi\\
& = \int_{\mathbb R^2} \Big(\frac{\rm i}{\kappa_{\rm p}} \frac{2\omega^2\beta}{|\xi|^2+\beta\gamma}(-\xi, \gamma)^{\top}g(\xi)\Big) \cdot \Big( \bar{A}_{\rm p}( \xi)( \xi, -\beta)^{\top}
 + \bar{\boldsymbol A}_{\rm s}( \xi) \Big){\rm d} \xi\\
& = \int_{| \xi|<\kappa_{\rm p}} \Big(\frac{\rm i}{\kappa_{\rm p}} \frac{2\omega^2\beta}{| \xi|^2+\beta\gamma}(- \xi, \gamma)^{\top}\Big) \cdot \Big( \bar{A}_{\rm p}( \xi)( \xi, -\beta)^{\top}
 + \bar{\boldsymbol A}_{\rm s}( \xi)\Big){\rm d} \xi\\
&=0.
\end{align*}
Similarly, in the case of \eqref{2.16}, we have
\begin{align*}
\mathcal F(v)
&= \int_{\mathbb R^2} \Big( \frac{\rm i}{\kappa_{\rm s}} \frac{2\omega^2\gamma}{| \xi|^2+\beta\gamma} \textbf{q}( \xi)^{\top}\times(\xi, -\beta)^{\top}\Big)\cdot \Big(\bar{A}_{\rm p}( \xi)(\xi, -\beta)^{\top}
+ \bar{\boldsymbol A}_{\rm s}( \xi) \Big) {\rm d} \xi\\
&=\int_{| \xi|<\kappa_{\rm s}} \Big( \frac{\rm i}{\kappa_{\rm s}} \frac{2\omega^2\gamma}{| \xi|^2+\beta\gamma} \textbf{q}( \xi)^{\top}\times( \xi, -\beta)^{\top}\Big)\cdot \Big(\bar{A}_{\rm p}( \xi)( \xi, -\beta)^{\top}\Big){\rm d}\xi  \\
&= 0.
\end{align*}
Therefore, the right-hand side of \eqref{4.9} is always orthogonal to each solution of \eqref{4.11}. Applying the Fredholm alternative, we obtain the existence of solutions to \eqref{4.9}.

Step 2. Let $v^{\rm sc}:=u-u^{\rm in}-u^{\rm re}$ in $\Omega$. Let $\tilde{v}^{\rm sc}$ be the zero extension of $v^{\rm sc}|_{\Lambda_0}$ onto $\Gamma_0$. Note that the sum of the incident field $u^{\rm in}$ and the reflected field $u^{\rm re}$ vanishes on $\Gamma_0^C$. We extend $v^{\rm sc}$ from $\Omega$ to $D$ by \eqref{uprc} with $b=0$, i.e.,
\begin{align}\label{ke}
 v^{\rm sc}(x)=\frac{1}{2\pi}\int_{\mathbb R^2}\Big\{
\frac{1}{\beta\,\gamma+| \xi|^2}  \Big(M_{\rm p}(\xi) e^{{\rm i} (\xi\cdot x^{\prime}+ \beta\, x_3)} + M_{\rm s}(\xi) e^{{\rm i} (\xi\cdot x^{\prime}+\gamma\, x_3)}  \Big)   \widehat{ \tilde{v}^{\rm sc}}(\xi, 0)\Big\} {\rm d} \xi,\quad x \in U_0.
\end{align}
We claim that the scattered field $v^{\rm sc}$ defined in \eqref{ke} can be represented as
\begin{align}\label{ex-vsc}
v^{\rm sc}(x) = \int_{\Gamma_0}T_y\text{G}_{\rm H}(x,y)v^{\rm sc}(y){\rm d}s(y),\quad x\in U_0,
\end{align}
where $\text{G}_{\rm H}(x,y)$ is the half-space Green's tensor (see \eqref{GH}) and
$T_y\text{G}_{\rm H}(x,y)$ represents the column-wisely action of the stress operator $T$ to $\text{G}_{\rm H}(x,y)$ with respect to the variable $y$. Since the trace of $v^{sc}$ on $\Gamma_0$ is compactly supported in $\Lambda_0$, by Lemma \ref{green}, $v^{\rm sc}$ satisfies the half-space Kupradze radiation condition, which completes the proof of the second part of Theorem \ref{exist}.

It remains to prove \eqref{ex-vsc}.
Since $v^{\rm sc}$ has compact support on $\Gamma_0$, applying the Fourier transform with respect to $y^{\prime}$ gives
\begin{align*}
\int_{\Gamma_0}T_y\text{G}_{\rm H}(x, y)v^{\rm sc}(y){\rm d}s(y) = \int_{\mathbb R^2}\widehat{T_y\text{G}_{\rm H}}(x, (-\xi, 0))\widehat{v^{\rm sc}}(\xi){\rm d}\xi.
\end{align*}
For simplicity of notation, we denote $\widehat{T_y\text{G}_{\rm H}}(x, (-\xi, 0))$ by $\widehat{T_y\text{G}_{\rm H}}(x, -\xi)$, which will be calculated as follows. By \eqref{GH},
\ben
\widehat{T_y\text{G}_{\rm H}}(x, -\xi)=
\widehat{T_y\text{G}}(x, -\xi)+\widehat{T_y\text{G}}(\tilde{x}, -\xi)
+\widehat{\text{U}}(x, -\xi).
\enn

 The Fourier transform of $\text{G}(x, y)$ with respect to the variable $y'$ on $\Gamma_0$ is
\begin{align*}
\hat{\text{G}}(x, \xi, 0) &= \frac{1}{\mu} \hat{g}_{\rm p}(x, \xi, 0)\text{I}\\
& \quad + \frac{(-\rm i)^2}{\omega^2}\hat{g}_{\rm p}(x, \xi, 0)
\begin{bmatrix}
\xi_1^2 & \xi_1\xi_2& \xi_1 \beta\\[5pt]
\xi_1 \xi_2 & \xi_2^2  & \xi_2 \beta\\[5pt]
\xi_{1}\beta           &\xi^2\beta  & \beta^2
\end{bmatrix}
- \frac{(-\rm i)^2}{\omega^2}\hat{g}_{\rm s}(x, \xi, 0)
\begin{bmatrix}
\xi_1^2 & \xi_1\xi_2& \xi_1 \gamma\\[5pt]
\xi_1 \xi_2 & \xi_2^2  & \xi_2 \gamma\\[5pt]
\xi_{1}\gamma           &\xi^2\gamma  & \gamma^2
\end{bmatrix}.
\end{align*}
The expression of $\hat{\text{G}}(\tilde{x}, \xi, 0)$ can be obtained analogously. For $x_3>0$, the functions $\text{G}(x, \cdot)$ and $\text{G}(\tilde{x}, \cdot)$ propagate downward and upward propagating near $\Gamma_0$, respectively. It follows from the downward and upward DtN maps that
\begin{align}\label{ty1}
\widehat{T_y\text{G}}(x, \xi) = {\rm i}M^-(\xi)\hat{\text{G}}(x, \xi, 0),\quad
\widehat{T_y\text{G}}(\tilde{x}, \xi) = {\rm i} M(\xi)\hat{\text{G}}(\tilde{x}, \xi, 0),
\end{align} where the matrices $M$ and $M^-$ are given by (\ref{M}) and \eqref{Mn}, respectively.
Moreover, we have from \eqref{U} that
\begin{align}\label{ty3}
\widehat{T_y\text{U}}(x, \xi) = \frac{\rm i}{2\pi\omega^2}\frac{e^{-\rm i \xi \cdot x^{\prime}}}{\beta\,\gamma+| \xi|^2}
 \Big(T_{\rm p}(\xi)\widetilde{M}_{\rm p}(\xi) + T_{\rm s}(\xi)\widetilde{M}_{\rm s}(\xi)\Big)(e^{\rm i \beta x_3} - e^{\rm i \gamma x_3}),
\end{align}
where
\begin{align*}
T_{\rm p}(\xi) := \rm i
\begin{bmatrix}
\mu \beta & 0 & \mu \xi_1\\[5pt]
0 & \mu \beta  &\mu \xi_2\\[5pt]
\lambda \xi_1          &\lambda \xi_2  & (\lambda + 2\mu)\beta
\end{bmatrix},\quad
T_{\rm s}(\xi) := \rm i
\begin{bmatrix}
\mu \gamma & 0 & \mu \xi_1\\[5pt]
0 & \mu \gamma  &\mu \xi_2\\[5pt]
\lambda \xi_1          &\lambda \xi_2  & (\lambda + 2\mu)\gamma
\end{bmatrix}.
\end{align*}
Combing \eqref{ty1}--\eqref{ty3}, we obtain after tedious but straightforward calculations that
\begin{align*}
&\widehat{T_y\text{G}}(x, -\xi) + \widehat{T_y\text{G}}(\tilde{x}, -\xi) + \widehat{T_y\text{U}}(x, -\xi)\\
&= {\rm i}\,M^-(-\xi)\hat{\text{G}}(x, -\xi, 0) +  {\rm i}\, M(-\xi)\hat{\text{G}}(\tilde{x}, -\xi, 0)\\
&\quad + \frac{\rm i}{2\pi\omega^2}\frac{e^{\rm i \xi \cdot x^{\prime}}}{\beta\,\gamma+| \xi|^2}
 \Big(T_{\rm p}(-\xi)\widetilde{M}_{\rm p}(-\xi) + T_{\rm s}(-\xi)\widetilde{M}_{\rm s}(-\xi)\Big)(e^{\rm i \beta x_3} - e^{\rm i \gamma x_3})\\
&=
\frac{1}{\beta\,\gamma+| \xi|^2}  \Big(M_{\rm p}(\xi) e^{{\rm i} (\xi\cdot x^{\prime}+ \beta\, (x_3-b))} + M_{\rm s}(\xi) e^{{\rm i} (\xi\cdot x^{\prime}+\gamma\, (x_3-b))}\Big).
\end{align*}
Furthermore, we obtain from \eqref{ke} that
\begin{align*}
v^{\rm sc}(x)=\int_{\mathbb R^2}\widehat{T_y\text{G}_{\rm H}}(x, -\xi)\widehat{v^{\rm sc}}(\xi){\rm d}\xi =
\int_{\Gamma_0}T_y\text{G}_{\rm H}(x,y)v^{\rm sc}(y){\rm d}s(y),
\end{align*}
which completes the proof of (\ref{ex-vsc}).
\end{proof}

\begin{rema}\label{Rem:5.4}

We make a few comments on the existence result in Theorem \ref{exist}.

(i) If $u^{\rm in}$ is of the form (\ref{2.15}), then the reflected wave $u^{\rm re}_{\rm pg}$ is given by (cf. \eqref{re-p})
    \begin{align}
u^{\rm re}_{\rm pg}(x)&= -\int_{\mathbb R^2} \frac{ (\xi,\gamma)^\top\cdot (\xi,-\beta)^\top}{(\beta\gamma+|\xi|^2)^2}M_{\rm p}(\xi)(\xi, \beta)^\top g(\xi) e^{{\rm i} (\xi\cdot x^{\prime}+ \beta (x_3-b))}{\rm d}\xi \notag\\
&\quad - \int_{\mathbb R^2}\frac{1}{(\beta\gamma+|\xi|^2)^2} M_{\rm p}(\xi)\Big[(\xi, \gamma)^\top\times \Big((\xi,-\beta)^\top\times(\xi, \beta)^\top\Big)\Big] g(\xi) e^{{\rm i}(\xi\cdot x^{\prime}+\gamma (x_3-b))}{\rm d}\xi. \notag
    \end{align}
    If $u^{\rm in}$ is of the form (\ref{2.16}), then the reflected wave $u^{\rm re}_{\rm sg}$ is given by (cf. \eqref{re-s})
    \begin{align}
    u^{\rm re}_{\rm sg}(x)&=
-\int_{\mathbb R^2} \frac{ (\xi,\gamma)^\top\cdot ((\xi,-\beta)^\top\times \textbf{q}(\xi))}{(\beta\gamma+|\xi|^2)^2}M_{\rm s}(\xi)(\xi, \beta)^\top e^{{\rm i} (\xi\cdot x^{\prime}+ \beta(x_3-b))}{\rm d}\xi \notag\\
&\quad - \int_{\mathbb R^2}\frac{1}{(\beta\gamma+|\xi|^2)^2} M_{\rm s}(\xi)\Big[(\xi, \gamma)^\top\times \Big(((\xi,-\beta)^\top\times \textbf{q}(\xi))\times(\xi, \beta)^\top\Big)\Big]e^{{\rm i} (\xi\cdot x^{\prime}+\gamma (x_3-b))}{\rm d}\xi. \notag
    \end{align}
Thus, if $u^{\rm in}$ takes the general form \eqref{incc}, it follows from the linear superposition that the reflected wave is given by
\begin{align*}
{u}^{\rm re}(x) = c_{\rm p}{u}^{\rm re}_{\rm pg}(x) +
c_{\rm s}{u}^{\rm re}_{\rm sg}(x).
\end{align*}

(ii) It is unclear whether the solution given by Theorem \ref{exist} is unique or not. By the proof of Theorem \ref{uniqueness}, the uniqueness is correct if the third component of the normal at the boundary $S$ is non-negative (i.e., $\nu_3\geq 0$). Note that this condition covers interfaces given by step functions and is thus weaker than the assumption used in Section \ref{uniqueness}. For the Helmholtz and Maxwell equations,
the well-posedness results have been established for general locally perturbed flat surfaces which are not necessarily the graph of a function (see \cite{Wood2006,lwz12, Li2010}). The arguments rely heavily on properties of the DtN maps derived from the corresponding reflection principle.
However, due to the lack of a pointwise reflection principle for the first boundary value problem of the Navier equation, we are not sure whether the DtN approach can be applied to our scattering problem. Thus, we can only obtain the existence result in the general case.

(iii) The result in Theorem \ref{exist} improves the acoustic and electromagnetic counterparts in the following sense. First, it shows that the existence results can be verified for general incoming waves from the upper half space even if the uniqueness is unknown. One can expect the same conclusion for acoustic and electromagnetic transmission problems. Second, the split of $u^{sc}$ into the sum $u^{\rm re}+v^{\rm sc}$ was rigorously justified under the mild assumption that $u^{sc}$ satisfies the UASR (\ref{uprc}).

\end{rema}

\subsubsection{Case (ii): perturbation above the ground plane}

In this subsection, we consider the scattering surface $\Gamma=\{ x\in\mathbb{R}^3: x_3=f(x^{\prime}),\,x^{\prime}\in\mathbb{R}^2\}$, where $f$ is a Lipschitz continuous function and is assumed to satisfy $f(x^{\prime})=0$ when $|x^{\prime}|>R$ for some $R>0$. This means that $\Gamma$ is a local perturbation of the ground plane $x_3 = 0$. The problem geometry is shown in Figure \ref{fig1}. Let $D=\{x\in\mathbb{R}^3: x_3>f(x^{\prime}),\,x^{\prime} \in \mathbb R^2\}$ and $\Lambda_R:= \Gamma \cap \{x: |x^{\prime}|\leq R\},$ which contains the perturbed part of $\Gamma.$ Denote by $\Omega_R = \{x \in D: |x|< R\}$ the truncated bounded domain and by $B_R^+ = \{ x \in \mathbb R^3: |x|<R, ~ x_3>0\}$ the upper half sphere. Let $S_R = \{x \in D: |x| = R \}$ and denote by $\nu$ the unit normal vector on $S_R$, pointing into the exterior of $\Omega_R$. Obviously, $\partial \Omega_R = \Lambda_R \cup S_R$.

\begin{figure}
\centering
\includegraphics[width=7cm,viewport=0.pt 100.pt 800.pt
  700.pt,clip]{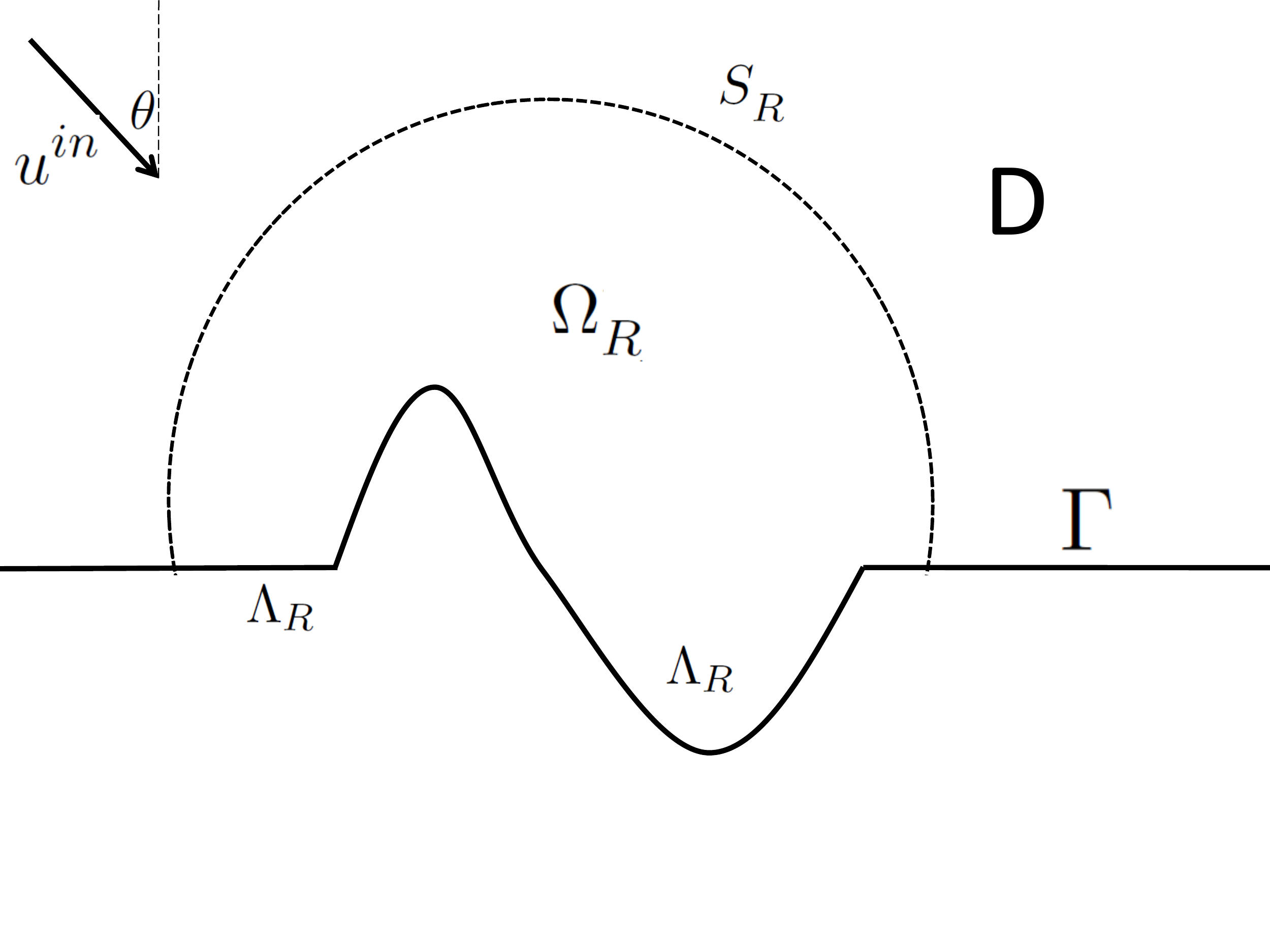}
\caption{Geometry of the scattering problem in a locally perturbed half plane. }
\label{fig1}
\end{figure}

Let $u^{\rm in}$ be the incident elastic plane wave  \eqref{inc}. %which satisfies the Navier equation \eqref{ifne}. The displacement of the total field $ u$ also satisfies the same Navier equation:
Due to the local perturbation, we suppose that
 the scattered field $u^{\rm sc}= u^{\rm re} +  v^{\rm sc}$ can be further decomposed into the sum of the reflected wave $ u^{\rm re}$ and $ v^{\rm sc}$,
where $ u^{\rm re}$ is the reflected field of the form \eqref{reflect} solving the unperturbed scattering problem
and $v^{\rm sc}$ satisfies the outgoing Kupradze radiation condition as defined in Definition \ref{def:radiation}.

Define the Sobolev space
$X_R= \{v \in H^1(\Omega_R)^3: v =0 ~\text{on} ~\Lambda_R\}$ and denote by $X_R^{-1}$ the dual space of $X_R$.
Introduce the Sobolev spaces on the open surface (see e.g., \cite{M-CUP00}):
\[\begin{split}
H^{1/2}(S_R)^3:=\{u|_{S_R}: u\in H^{1/2}(\partial\Omega_R)^3\},\quad
\tilde{H}^{1/2}(S_R)^3:=\{u\in H^{1/2}(\partial \Omega_R)^3: \mbox{supp} (u) \subset S_R\}.
\end{split}\]
Denote by $H^{-1/2}(S_R)^3$ the dual space of $\tilde{H}^{1/2}(S_R)^3$ and by $\tilde{H}^{-1/2}(S_R)^3$ the dual space of $H^{1/2}(S_R)^3$.

Next, we introduce the generalized stress (or traction) operator and the corresponding bilinear form
\begin{align*}
T_{a,b}u &= (\mu + a)\partial_{ \nu} u + b  \nu \nabla \cdot  u + a \nu \times (\nabla \times  u), \\
\mathcal{E} (u, w) &= (\mu + a)\sum_{j,k=1}^3 \partial_k u_j \partial_k w_j + b (\nabla \cdot u)(\nabla \cdot w) - a (\nabla \times u) \cdot (\nabla \times w)
\end{align*}
where %$\boldsymbol \nu = (\nu_1, \nu_2, \nu_3)$ denotes the unit normal directed into the exterior of $\Omega_R$ and
$a,b\in \R$ satisfying $a + b = \lambda + \mu$. Throughout this section we choose
\begin{align*}
a = \frac{\mu (\lambda + \mu)}{\lambda + 3 \mu}, \quad b = \frac{(\lambda + \mu) (\lambda + 2\mu)}{\lambda + 3 \mu}.
\end{align*}
The above choice of $a$ and $b$ yields a compact double layer operator $\mathcal D$ with a weakly singular kernel (see \cite{hahner}) as defined below in \eqref{sd}.
For simplicity we still denote $T_{a, b}$ by $T_{\nu}$, which is called the pseudo stress operator \cite{hahner} with the new choice of $a$ and $b$. Note that the usual stress operator corresponds to $a=\mu$ and $b=\lambda$ and the Betti's formula are still valid for the new choice, i.e.,
\begin{align}\label{vf}
\int_{\Omega_R} \mathcal{E} ( u, v) - \omega^2  u \cdot v\; {\rm d}x - \int_{S_R}  v \cdot T_{\nu} u {\rm d}s = 0.
\end{align}

By applying Green's formula and the half-plane Kupradze radiation condition, it is easy to derive the Green's representation formula for the scattered wave $v^{\rm sc}$:
\begin{align}\label{grf}
v^{\rm sc}(x) = \int_{S_R} T_{ \nu( y)}\text{G}_{\rm H}( x,  y) \cdot  v^{\rm sc}( y) - \text{G}_{\rm H}( x, y) \cdot T_{ \nu( y)} v^{\rm sc}( y) {\rm d}s(y), \quad  x \in D\backslash \overline{\Omega}_R.
\end{align}
Taking the limit $ x \rightarrow S_R$ in \eqref{grf} and setting $ p = T_{ \nu}  v^{\rm sc}|_{S_R} \in H^{-1/2}(S_R)^3$, we obtain
\begin{align}\label{jc}
(\frac{1}{2} \mathcal I - \mathcal D) ( v^{\rm sc}|_{S_R}) + \mathcal S  p = 0 \quad \text{on} ~S_R.
\end{align}
Here $\mathcal I$ is the identity operator, $\mathcal D$ and $\mathcal S$ are the double layer and single layer operators over $S_R$, respectively, defined by
\begin{align}\label{sd}
(\mathcal D  g)( x) = \int_{S_R} T_{ \nu(\boldsymbol y)}\text{G}_{\rm H}( x,  y)  g( y){\rm d}s( y),\quad
(\mathcal S  g)(\boldsymbol x) = \int_{S_R} \text{G}_{\rm H}( x,  y) g( y){\rm d}s( y).
\end{align}
Combing \eqref{vf} and \eqref{jc} yields the variational formulation for the unknown solution pair $( u, p) \in X_R \times H^{-1/2}(S_R)^3 := X$ as follows
\begin{align}\label{cvf}
\mathbb B(( u,  p), ( \varphi,  \chi)) =
\begin{bmatrix}
b_1(( u,  p), ( \varphi,  \chi))\\
b_2(( u,  p), ( \varphi,  \chi))
\end{bmatrix}
=
\begin{bmatrix}
\int_{S_R} T_{ \nu} u_0 \cdot \overline{ \varphi}{\rm d}s\\[5pt]
\int_{S_R} (\frac{1}{2} \mathcal I - \mathcal D) ( u_0|_{S_R}) \cdot \overline{ \chi}{\rm d}s
\end{bmatrix}
\end{align}
for all $(\varphi,  \chi) \in X$, where $ u_0 = u^{\rm in} +  u^{\rm re}$ is the reference field and
\begin{align*}
b_1(( u,  p), ( \varphi,  \chi)) &= \int_{\Omega_R} \mathcal{E} (u, \overline{ \varphi}) - \omega^2  u \cdot \overline{ \varphi} {\rm d}x - \int_{S_R}  \overline{ \varphi} \cdot  p {\rm d}s,\\
b_2(( u,  p), ( \varphi, \chi))&= \int_{S_R}\Big((\frac{1}{2} \mathcal I - \mathcal D) ( u|_{S_R}) + \mathcal S p \Big)\overline{ \chi} {\rm d}s.
\end{align*}

The Fredholm property of the sesquilinear form $\mathbb B$ can be proved by following almost the same lines in \cite{hyz18}. To prove the uniqueness, one has to assume that $\omega^2$ is not a Dirichlet eigenvalue of the operator $-(\mu \Delta+(\lambda+\mu)\nabla(\nabla\cdot))$ over $\Omega_R$. This assumption implies the equivalence of the variational problem (\ref{cvf}) posed on $\Omega_R$ and our scattering problem in $D$. As a consequence of Theorem \ref{uniqueness}, one obtains the uniqueness. We refer to \cite{hyz18} for the details and only state the well-posedness results below.

\begin{theo}\label{th5.12} Assume that $\omega^2$ is not a Dirichlet eigenvalue of the operator $-(\mu \Delta+(\lambda+\mu)\nabla(\nabla\cdot))$ over $\Omega_R$. Then, there exists a unique solution $u \in X_R$ to the variational formulation (\ref{cvf}). Moreover, one may extend $v^{\rm sc}:=u-u^{\rm in}-u^{\rm sc}$ from $\Omega_R$ to $D\backslash\overline{\Omega}_R$ through (\ref{grf}) and the extended solution satisfies the radiation solution (\ref{rc}).
\end{theo}

\begin{rema}\label{rema:5.7}
We make some comments on the well-posedness results in Theorem \ref{th5.12}.

(i)  In contrast with Theorem \ref{exist}, Theorem \ref{th5.12} is justified under the strong assumption that ${u}={u}^{\rm in} +  u^{\rm re} + v^{\rm sc}$ where $v^{\text{sc}}$ satisfies the half-plane Kupradze radiation condition; see \eqref{grf} where this assumption was used. This automatically implies that ${u}-{u}^{\rm in}$ fulfills the weaker radiation condition UPRC \eqref{uprc}. We refer to
Remark \eqref{Rem:5.4} (ii) for the reason why we cannot prove the uniqueness for non-graph scattering surfaces.

(ii) With the argument in the proof of Theorem \ref{th5.12}, one can discuss the well-posedness of the elastic scattering from a trapezoidal surface, which is a non-local perturbation of flat surfaces. This requires a modified radiating assumption on $u-u^{in}$ which depends on both the incident wave and the scattering surface; see \cite{LuHu2019} for the acoustic scattering problem with a trapezoidal sound-soft curve.

\end{rema}

Now we consider the boundary value problem in a locally perturbed half space:
\be\label{bvp}
\mu \Delta v+(\lambda+\mu)\nabla (\nabla\cdot v)+\omega^2 v=0\quad\mbox{in} ~ D,\quad
v=h\quad\mbox{on} ~ \Gamma,
\en
where $h\in (H^{1/2}(\Gamma))^3$ and $v$ is required to satisfy the UPRC (\ref{uprc}) in $x_3>0$. We can always find a function $h_0\in (H^{1/2}(\Gamma_0))^3$ such that $h_0=h$ in $\Gamma\cap\{x: |x|>R\}$ for the $R$ specified at the beginning of this subsection. Set
\ben
 v_0(x)=\frac{1}{2\pi}\int_{\mathbb R^2}\Big\{
\frac{1}{\beta\,\gamma+| \xi|^2}  \Big(M_{\rm p}(\xi) e^{{\rm i} (\xi\cdot x^{\prime}+ \beta\, x_3)} + M_{\rm s}(\xi) e^{{\rm i} (\xi\cdot x^{\prime}+\gamma\, x_3)}  \Big)   \hat{h}_0(\xi)
  \Big\} {\rm d} \xi,\quad x\in D.
\enn
Then $v_0\in H^1(\tilde{S}_b)^3$ for any $b>0$ with strip $\tilde{S}_b:=\{x: 0<|x_3|<b\}$ and it is an upward propagating solution to the Navier equation with the Dirichlet data $v_0=h_0$ on $x_3=0$. By Sobolev extension theorem (see e.g., \cite[Theorem 7.25]{GilTru-1983}), $v_0$ can be extended to a function $v_1\in H^1(S_b)$ from $x_3>0$ to $D$ such that $v_1\equiv v_0$ in $x_3>0$. Defining $w_1=v-v_1$, we deduce that
\ben
\mu \Delta w_1+(\lambda+\mu)\nabla (\nabla\cdot w_1)+\omega^2 w_1=f_1\quad\mbox{in} ~ D,\quad
w_1=h_1\quad\mbox{on} ~ \Gamma,
\enn
where $f_1\in (H^{1}(\Omega_R))^*$ is compactly supported in $D\cap\{x_3<0\}$ and $h_1\in H^{1/2}(\Gamma)$ is compactly supported in $\Lambda_R$.
Finally, by a lifting argument one can reduce the previous problem to a homogeneous boundary value problem
\ben
\mu \Delta w_2+(\lambda+\mu)\nabla (\nabla\cdot w_2)+\omega^2 w_2=f_2\quad\mbox{in} ~ D,\quad
w_2=0\quad\mbox{on} ~ \Gamma,
\enn
with $f_2\in (H^{1}(\Omega_R))^*$ compactly supported in $\Omega_R$, where $w_2=w_1 - v_2$ for some $v_2\in H^1(S_b)^3$ ($b>0$) such that $v_2\equiv h_1$ on $\Gamma$ and $v_2\equiv 0$ in $x_3>2R$.
Choose $R>0$ such that $\omega^2$ is not a Dirichlet eigenvalue of the operator $-(\mu \Delta+(\lambda+\mu)\nabla(\nabla\cdot))$ over $\Omega_R$. Then the above inhomogeneous source problem can be equivalently formulated as the variational problem:
\ben
\mathbb B(( w_2,  p), ( \varphi,  \chi))
=
\begin{bmatrix}
\int_{\Omega_R} f_2\cdot \overline{\varphi}\,  {\rm d}x\\[5pt]
0
\end{bmatrix},\quad p:=T_{\nu} w_2|_{S_R}\in H^{-1/2}(S_R)^3, \quad \forall (\varphi,  \chi) \in X.
\enn
By the proof of Theorem
\ref{th5.12}, there admits a unique solution $w_2\in H^1(\Omega_R)^3$, which can be extended to a Sommerfeld radiating solution in $D\cap\{|x|>R\}$. We summarize the solvability result as follows.
\begin{corollary}
The boundary value problem \eqref{bvp} admits a unique upward propagating solution $v=\tilde{v}+w_2\in H^1(S_b)^3$ for any $b>0$, where $\tilde{v}$ satisfies the UASR (\ref{uprc}) and $w_2$ satisfies the half-space Kupradze radiation condition.
\end{corollary}

\section{Concluding remarks}\label{sec:con}
\setcounter{equation}{0}

We have presented the mathematical formulation of time-harmonic elastic scattering from general unbounded rough surfaces in three dimensions. In particular, the ASR in a half space is derived and properties of the DtN map are analyzed. The uniqueness is proved for the Lipschitz continuous rough surface which is given by the graph of a function. We deduce the Green's tensor for the first boundary value problem of the Navier equation in a half space. The existence of weak solution to locally perturbed scattering problem is established by applying the Fredholm alternative to an equivalent variational formulation in a truncated bounded domain.

Below we list three interesting questions for locally perturbed scattering problems which deserve to be further investigated:
  \begin{itemize}
  \item The uniqueness result for perturbations given by non-graph functions.

  \item Equivalent variational formulation in a bounded domain without the coupling scheme between the finite element method
and the integral representation. In particular, a numerical scheme avoiding the half-space Green's tensor and involving the free-space's tensor only would be desirable from the numerical point of view.

  \item Explicit dependence of the solution on the frequency of incidence in linear elasticity. The variational approach developed \cite{cm-siaa05} leads to an explicit wave-number dependence of solutions to the acoustic rough surface scattering problems. However, the derivation of such kind of estimates relies on the positivity of the real part of the DtN map (see \cite[Lemma 3.2]{cm-siaa05}), which unfortunately is not applicable to the Navier equation.
  \end{itemize}

Based on the framework presented in this work, we plan to carry out the study of the elastic scattering from globally perturbed (non-periodic) rough surfaces, for example, due to an inhomogeneous elastic source term or an incoming point source incidence. This will extend at least the acoustic results of \cite{cm-siaa05} and \cite{CWEL} in weighted and non-weighted Sobolev spaces to linear elasticity in three dimensions. In particular, the absence of elastic surfaces can be proved as a consequence of well-posedness in weighted Sobolev spaces.

\end{document}